\pdfoutput=1
\RequirePackage{ifpdf}
\ifpdf 
\documentclass[pdftex]{sigma}
\else
\documentclass{sigma}
\fi

\numberwithin{equation}{section}

\newtheorem{Theorem}{Theorem}[section]
\newtheorem*{Theorem*}{Theorem}
\newtheorem{Corollary}[Theorem]{Corollary}
\newtheorem{Lemma}[Theorem]{Lemma}
\newtheorem{Proposition}[Theorem]{Proposition}
 { \theoremstyle{definition}
\newtheorem{Definition}[Theorem]{Definition}

\newtheorem{Example}[Theorem]{Example}
\newtheorem{Examples}[Theorem]{Examples}
\newtheorem{Remark}[Theorem]{Remark}
\newtheorem{Remarks}[Theorem]{Remarks} }

\newcommand{\R}{\ensuremath{\mathbb{R}}}
\newcommand{\N}{\ensuremath{\mathbb{N}}}
\newcommand{\Z}{\ensuremath{\mathbb{Z}}}

\newcommand{\C}{\ensuremath{\mathbb{C}}}
\newcommand{\T}{\ensuremath{\mathbb{T}}}

\def\calT{\mathcal{T}}
\def\calC{\mathcal{C}}

\def\calK{\mathcal{K}}
\def\calB{\mathcal{B}}
\def\calH{\mathcal{H}}

\def\calA{\mathcal{A}}

\def\calQ{\mathcal{Q}}

\def\calU{\mathcal{U}}

\newcommand{\wh}{\ensuremath{\widehat}}
\newcommand{\wt}{\ensuremath{\widetilde}}
\DeclareMathOperator{\End}{End}

\newcommand{\one}{{\bf 1}}
\newcommand{\ol}{\overline}
\DeclareMathOperator{\Aut}{Aut}

\theoremstyle{definition}

\DeclareMathOperator{\Dom}{Dom}

\DeclareMathOperator{\Index}{Index}
\DeclareMathOperator{\Ind}{Ind}
\DeclareMathOperator{\Mult}{Mult}

\DeclareMathOperator{\Ker}{Ker}

\DeclareMathOperator{\Ran}{Ran}
\DeclareMathOperator{\Tr}{Tr}

\newcommand{\Cl}{\mathbb{C}\ell}

\newcommand{\hox}{ \hat{\otimes} }

\begin{document}
\allowdisplaybreaks

\newcommand{\arXivNumber}{2211.10601}

\renewcommand{\PaperNumber}{053}

\FirstPageHeading

\ShortArticleName{Index Theory of Chiral Unitaries and Split-Step Quantum Walks}

\ArticleName{Index Theory of Chiral Unitaries\\ and Split-Step Quantum Walks}

\Author{Chris BOURNE~$^{\rm ab}$}

\AuthorNameForHeading{C.~Bourne}

\Address{$^{\rm a)}$~Institute for Liberal Arts and Sciences and Graduate School of Mathematics,\\
\hphantom{$^{\rm a)}$}~Nagoya University, Furo-cho, Chikusa-ku, Nagoya, 464-8601, Japan}

\Address{$^{\rm b)}$~RIKEN iTHEMS, 2-1 Hirosawa, Wako, Saitama 351-0198, Japan}
\EmailD{\href{mailto:bourne@ilas.nagoya-u.ac.jp}{cbourne@ilas.nagoya-u.ac.jp}}

\ArticleDates{Received December 04, 2022, in final form July 24, 2023; Published online July 28, 2023}

\Abstract{Building from work by Cedzich et al.\ and Suzuki et al., we consider topological and index-theoretic properties of chiral unitaries, which are an abstraction of the time evolution of a chiral-symmetric self-adjoint operator. Split-step quantum walks provide a rich class of examples. We use the index of a pair of projections and the Cayley transform to define topological indices for chiral unitaries on both Hilbert spaces and Hilbert $C^*$-modules. In the case of the discrete time evolution of a Hamiltonian-like operator, we relate the index for chiral unitaries to the index of the Hamiltonian. We also prove a double-sided winding number formula for anisotropic split-step quantum walks on Hilbert $C^*$-modules, extending a result by Matsuzawa.}


\Classification{46L80; 47L90}

\section{Introduction}

Quantum walks are discrete analogues of random walks and bring together elements of probability
theory, quantum information theory and mathematical physics.
Such quantum walks are often
realised via a unitary operator, considered as a discrete time step,
constructed from the discrete shift and multiplication/coin operators on the
Hilbert space $\ell^2(\Z, \C^n)$.

Adapting ideas from topological phases
of matter, Kitagawa examined topological properties of \emph{split-step} quantum walks, which are a
sub-class of quantum walks whose spectrum is symmetric about the real axis~\cite{Kitagawa12}.
This was then extended by Cedzich et al., who gave a~mathematical classification of unitaries $U$
on a Hilbert space $\calH$ with $\pm 1$ not in the essential spectrum and obeying
symmetry relations adapted from free-fermionic topological insulators and superconductors~\cite{Cedzich18, Cedzich22}.
The case $U = \Gamma_0 \Gamma_1$ where $\Gamma_0$ and $\Gamma_1$ are self-adjoint unitaries
was further examined by Suzuki, who studied the index by Cedzich et al.
from the perspective of supersymmetry and the Fredholm index~\cite{Suzuki19}.
Throughout this paper, we will call unitaries with a decomposition as a product of self-adjoint unitaries
\emph{chiral unitaries}.
Our terminology comes from the property $\Gamma_0 U \Gamma_0 = U^*$, which can be considered
as a chiral symmetry of the unitary $U$.

The total symmetry index for chiral unitaries on a Hilbert space studied by Cedzich et al.\ and Suzuki possesses two important properties.
First, it is
realised via the index of a supersymmetric (odd) Fredholm operator. That is, for a given $U = \Gamma_0 \Gamma_1$,
the self-adjoint Fredholm operator
$\tfrac{1}{2\rm{i}}(U-U^*)$
anti-commutes with the self-adjoint unitary $\Gamma_0$, which can be considered as a grading operator of the Hilbert space.
Second, and more importantly, the index of this Fredholm operator is closely connected to invariant
and quasi-invariant states of the discrete time step $U$. Namely, a non-trivial index guarantees the existence
of eigenvectors in either $\Ker(U-\one)$ or $\Ker(U+\one)$.

Analogues of the symmetry index for chiral unitaries on Hilbert spaces have been generalised in a variety of areas, including
non-unitary split-step quantum walks~\cite{nonUnitaryQW} and situations where the Fredholm condition
fails, where the Witten index is employed instead~\cite{QWWitten21}. The purpose of this paper is to
give a new generalisation using tools from $K$-theory and (noncommutative) index theory. Namely, we
broaden our definition of chiral unitaries to include adjointable operators on Hilbert $C^*$-modules.
In this generalised setting, we obtain topological indices taking value in the
$K$-theory group of a $C^*$-algebra. The Hilbert space setting can be recovered by considering
Hilbert $C^*$-modules over $\C$.

 The condition we require in order to define topological indices for chiral unitaries is purely spectral
(or a related condition in the $C^*$-module setting using the generalised Calkin algebra) and so is independent from the setting
of quantum walks. However, the class of split-step quantum walks provides us with a rich class of examples that are of relevance
in a variety of fields. So while our constructions are general, they are done with
concrete applications in mind.

Let us provide some motivations to study chiral unitaries on Hilbert $C^*$-modules. Suppose that
we have a compact Hausdorff parameter space $X$ for which we can consider a chiral unitary
$U_x \in \calU(\calH)$ for any $x \in X$. If the map $X\ni x \mapsto U_x \in \calB(\calH)$
is strongly continuous, then the family $u = \{U_x\}_{x\in X}$ defines a unitary operator on the Hilbert
$C^*$-module $(\calH \otimes C(X))_{C(X)}$ (see Section \ref{subsec:C*_modules} for precise
definitions). Therefore, we can consider index-theoretic
properties of the family $u = \{U_x\}_{x\in X}$ that will also be sensitive to the topology of the parameter space~$X$.
Another motivation is that topological indices that take value in the $K$-theory group of a~$C^*$-algebra
are in general much richer than the setting of operators on Hilbert spaces. Hence, we gain access to a
wider variety of possible invariants that can not be accessed by working solely on Hilbert spaces.
Indeed, the Fredholm-like condition required to define $K$-theoretic indices is substantially weaker than
the case of
Hilbert spaces. As an example from topological phases of matter, suppose $H=H^*$ is a Hamiltonian acting on
$\ell^2\bigl(\Z^d, \C^n\bigr)$, $d\geq 2$, and constructed from matrices of polynomials of the discrete shift operators
on $\ell^2\bigl(\Z^d, \C^n\bigr)$ such that $0 \notin \sigma(H)$. Then the restriction of $H$ to the half-space
$\ell^2\bigl(\Z^{d-1}\times \N, \C^n\bigr)$ will define
a Fredholm operator on the Hilbert $C^*$-module $\bigl(\ell^2(\N, \C^n) \otimes M_n\bigl( C^*\bigl(\Z^{d-1}\bigr) \bigr)\bigr)_{M_n( C^*(\Z^{d-1}) )}$
with $C^*\bigl(\Z^{d-1}\bigr)$ the group $C^*$-algebra of~$\Z^{d-1}$. Hence, Hilbert $C^*$-modules also give us a mathematical
framework to study topological properties of edge/boundary phenomena, which may be useful for applications
in physics.\looseness=-1

The main tools we use to define topological indices are the Cayley transform of unitary
operators and the index of a pair of projections. That is, given a chiral unitary
$U$, we may consider the index of the (generally unbounded) self-adjoint operator
${\calC(U) = {\rm i}( \one +U)(\one-U)^{-1}}$
or~$\calC(-U)$. Similarly, writing $U = \Gamma_0 \Gamma_1 = (2P_0-\one)(2P_1-\one)$ with
$P_0$ and $P_1$ self-adjoint projections, we may consider the index of a pair of projections
$\Ind(P_0, P_1)$ or $\Ind(P_0, \one - P_1)$. If these indices are well defined, their sum
$\Index( \calC(U) ) + \Index( \calC(-U) )$ or $\Ind(P_0, P_1) + \Ind(P_0, \one - P_1)$ recovers
the previously studied total symmetry index for chiral unitaries,
 though the individual indices are defined more generally. Indeed,
$\Index( \calC(U) )$ and $\Ind(P_0, P_1)$ examine $\Ker(U+\one)$ while
$\Index( \calC(-U) )$ and $\Ind(P_0, \one - P_1)$ examine $\Ker(U-\one)$. The relation between
the index for chiral unitaries and the sum of an index of a pair of projections was already noted in~\cite{Suzuki19}.
One advantage of working with the Cayley transform is that $\Index( \calC(U) )$
 still comes from the index of a self-adjoint odd Fredholm
operator as $\calC( \pm U)$ anti-commutes with $\Gamma_0$. The cost is that we need to work
with unbounded operators. We make the above statements
precise for unitaries acting on general Hilbert $C^*$-modules over an auxiliary
$C^*$-algebra.\footnote{Some care is needed to discuss the index of a pair of projections in a Hilbert
$C^*$-module as closed subspaces need not be complemented.}

Outside of the quantum walk setting, our constructions provide a new presentation
of the $K$-theory group $K_0(A)$ of a $C^*$-algebra $A$ in terms of equivalence classes of chiral unitaries acting on
Hilbert $C^*$-modules over $A$, which may be of independent interest.

{\samepage An important class of chiral unitaries arise from the discrete time evolution of a Hamiltonian-like
operator ${\rm e}^{{\rm i} \pi H}$, where $H$ possesses a~so-called chiral symmetry, $\Gamma_0 H \Gamma_0 = -H$
with $\Gamma_0$ a~self-adjoint unitary.\footnote{Indeed, the example of the time-evolution of a chiral-symmetric
Hamiltonian is the main motivation for our use of the term `chiral unitaries'.}
If $H$ is Fredholm, then we can consider the symmetry index of $H$ directly. Up to a regularisation of
$H$, we show that this index can be computed by the index of $\calC\bigl( - {\rm e}^{{\rm i} \pi H} \bigr)$ or
$\Ind(P_0, \one-P_1)$, which we also
extend to the setting of Hilbert $C^*$-modules and $K$-theoretic indices.

}

As a brief application, we consider split-step quantum walk-like chiral unitaries on the
Hilbert $C^*$-module $\ell^2(\Z, B)$, where $B$ is a unital $C^*$-algebra with a trace. Restricting to
operators constructed from the shift on $\Z$ and functions $\Z \to B$ with limits at $\pm \infty$,
a short exact sequence of $C^*$-algebras can be constructed that is closely related to the Toeplitz extension of the
crossed product $B \rtimes \Z$. Making this connection precise,
we can then use a semifinite analogue of the Noether--Toeplitz index theorem to relate the trace of the topological index of
a chiral unitary to the difference of winding numbers of a related unitary at the limit points $\pm \infty$.
Such a result was previously shown by Matsuzawa in the Hilbert space setting~\cite{Matsuzawa}.

Another advantage of working with unitaries on Hilbert $C^*$-modules is that non-trivial $K$-theoretic indices
can also be defined for regular (non-split-step) quantum walks. While the lack of chiral symmetry means there
is no $\Z_2$-grading,
the Cayley transform of generic unitaries can still be used
to obtain a self-adjoint operator. If the Cayley transform is Fredholm, topological indices can be constructed
in odd $K$-theory.

We have previously studied the Cayley transform as a convenient tool to pass between self-adjoint and unitary
operators that is compatible with $K$-theory~\cite{B21, BKR2}. Indeed, our results will not be surprising to
those familiar with noncommutative index theory. Rather, we hope that this paper can be used as a starting
point to further probe the connections of index theory with quantum walks.

We note that while the works~\cite{GrafTauber, LHR18, SadelSchuba} also adopt techniques from the study
of topological phases to quantum walks, the setting of the current paper is rather different. In particular, we do
not consider bulk-edge phenomena, leaving this question for future work.

During the writing of this paper, we became aware of similar work by Moriyoshi and Natsume, who also
prove a Noether-like index theorem for split-step quantum walks~\cite{MoriyoshiNatsume}.

\subsection{Outline and main results}

Sections \ref{sec:HSpace_chiral_unitaries} and \ref{sec:NCIndex_theory_prelims} serve as a summary of previous
results and, while some proofs are provided for completeness, no originality is claimed for any results in this part of the paper.

Section \ref{sec:HSpace_chiral_unitaries} considers $\Z$-valued indices defined for chiral unitaries on
Hilbert spaces. This section is mostly self-contained and independent from the rest of the paper.
 We give a survey of the various topological indices for chiral unitaries on Hilbert spaces that have appeared
in the mathematical physics literature and show their connections to one another.
We then give another presentation of these indices via the index of a pair of projections and show
some basic stability properties and trace formulas. The connection to the generator of the
unitary time step and the Cayley transform is also briefly considered.

Section \ref{sec:NCIndex_theory_prelims} is an executive summary of Hilbert $C^*$-modules and $K$-theory
via Fredholm operators, mostly following~\cite{Wahl07}.
In particular we use a picture of $K$-theory that comes
about via the isomorphisms $K_0(A) \cong KK(\C, A)$ and $K_1(A) \cong KK^1(\C, A)$ for a $\sigma$-unital
$C^*$-algebra~$A$. We also consider an analogue of the index of a pair of projections in Hilbert $C^*$-modules,
where due to the lack of guaranteed complementability of subspaces, we instead construct an index via a~concrete Kasparov module and equivalence class in $KK(\C, A)$. We show some basic properties of
this index and an
addition formula, which helps justify our terminology.
Lastly in the case of a unital $C^*$-algebra $B$ with trace $\tau$, we review how to extract a numerical index via the map~$\tau_\ast\colon KK(\C,B) \to \R$.

In Section \ref{sec:Cstar_mod_index}, we return to chiral unitaries now defined on Hilbert $C^*$-modules.
Our first result is a precise condition on chiral unitaries that ensures the Cayley transform is a Fredholm
operator on a Hilbert $C^*$-module $E_A$ and defines
an even $K$-theory class for the algebra $A$.
We can also work with arbitrary chiral unitaries in $\operatorname{Mult}(A)$, the multiplier algebra
of $A$, by considering $A$ as a $C^*$-module over itself.
We also introduce the Cayley transform and $K$-theoretic index for a pair of projections that are
sufficiently close in the generalised Calkin algebra. We show that our indices for chiral unitaries and
pairs of projections are compatible with
each other as well as more standard presentations of the $K_0$-group of a $C^*$-algebra.
 Properties of our $K$-theoretic index analogous to those in
Section~\ref{sec:HSpace_chiral_unitaries} are also shown, though the proofs require more operator algebraic
machinery. In particular, we are able to relate the $K$-theoretic index of a self-adjoint and odd Fredholm operator $H$
with $\|H \| \leq 1$ to the index defined from the chiral unitary ${\rm e}^{{\rm i} \pi H}$. As previously mentioned in this
introduction, this example is of particular relevance for the case of a $H$ a Hamiltonian on a system with
boundary.
Interestingly, many of our results in this section also hold true for non-chiral unitaries up to a
$K$-theoretic degree shift
and so may be applicable to a wider class of quantum walk
systems.

Finally in Section \ref{sec:1D_winding}, we prove a noncommutative version of the Noether--Toeplitz index theorem
for the index of chiral unitaries $u$ acting on the Hilbert $C^*$-module $\ell^2(\Z, B)$ with an anisotropic condition
on the element $u$. Applying the trace map to $[\calC(u)] \in KK(\C, B)$ gives us a numerical quantity and our index
formula shows that this number can be computed via the
noncommutative winding number of pair of unitaries in $B\rtimes \Z$ that are defined at the limits $\pm \infty$ of the
space $\ell^2(\Z, B)$. While still more abstract than the winding number on Hilbert spaces, the noncommutative
winding number is generally much easier to compute than $K$-theoretic Fredholm indices.

Our work deals extensively with projections and unitaries acting on both Hilbert spaces and
Hilbert $C^*$-modules. To help distinguish between these settings, we will use capital letters
$U$ and $P$ for unitaries and projection on Hilbert spaces and small letters $u$ and $p$ for
unitaries and projections in more general $C^*$-algebras.

\section{Chiral unitaries on Hilbert spaces} \label{sec:HSpace_chiral_unitaries}

In this section, we review the various $\Z$-valued indices that have been
defined for unitary operators $U$ on a
separable Hilbert space $\calH$ with $\pm 1\notin \sigma_{\rm ess}(U)$ and whose spectrum is symmetric about the real axis.
We then relate these indices both to each other and
the index of a pair
of projections popularised by Avron, Seiler and Simon~\cite{ASS}. Connections to the
index of a self-adjoint generator and the Cayley transform for unitaries are also considered,
which helps motivate later sections.
Much of the content of this section can be directly
found or easily adapted from analogous results in~\cite{Cedzich18, MST21, Suzuki19}.

\subsection{Basic definitions}

We denote by $\calB(\calH)$ and $\calK(\calH)$ the algebra of bounded and compact operators on
$\calH$, respectively. The quotient $\calQ(\calH) = \calB(\calH)/\calK(\calH)$ is the Calkin algebra.
Given $S \in \calB(\calH)$, we will often abuse notation and write $\| S \|_{\calQ(\calH)}$ to mean
$\| \pi(S) \|_{\calQ(\calH)}$ with $\pi\colon \calB(\calH) \to \calQ(\calH)$ the quotient map.

We first note an elementary result.

\begin{Lemma} \label{lem:chiral_iff_decomp}
Let $B$ be a unital $C^*$-algebra, $u \in B$ a unitary element and $\gamma_0 \in B$ a self-adjoint unitary.
Then $\gamma_0 u \gamma_0 = u^*$ if and only if there is a self-adjoint unitary $\gamma_1 \in B$ such that
$u = \gamma_0 \gamma_1$.
\end{Lemma}
\begin{proof}
The right-to-left implication is a simple check. Similarly, if $\gamma_0 u \gamma_0 = u^*$, we define $\gamma_1 = \gamma_0 u$
and check that $\gamma_1^* = u^* \gamma_0 = \gamma_0 u$.
\end{proof}

\begin{Definition}
We say that a unitary operator $U \in \calB(\calH)$ is a chiral unitary
 if $U = \Gamma_0 \Gamma_1$, where $\Gamma_j$ are self-adjoint unitaries on $\calH$ for $j \in \{0,1\}$.
\end{Definition}

Because $\Gamma_0 U \Gamma_0 = U^*$, the spectrum of the chiral unitaries is symmetric about the real axis.
This observation along with the example of a discrete time evolution of a chiral-symmetric Hamiltonian
(see Section \ref{subsec:H_space_generator}) motivates our use of the term `chiral unitary'.

\begin{Example}[split-step quantum walks] \label{ex:canonical_example}
An important class of examples of chiral unitaries arise from split-step (discrete) quantum walks, which for
the case of the lattice $\Z$ are often
 constructed from shift and coin operators on $\calH = \ell^2\bigl(\Z, \C^2\bigr)$. We
fix functions $a\colon\Z\to \R$, $b\colon\Z \to \C$,
$n \in \N$ as well as $n$-periodic and bounded functions $c\colon \Z\to \R$ and $d\colon \Z \to \C$ such
that $a(x)^2 + |b(x)|^2 = c(x)^2 + |d(x)|^2 = 1$ for all $x \in \Z$.
Then using the
shift operator $(S\psi)(x) = \psi(x-1)$ for $\psi \in \ell^2(\Z)$ and $x \in \Z$, we define $U = \Gamma_0 \Gamma_1$ with
\[
 \Gamma_0 = \begin{pmatrix} c & (S^n)^* \ol{d} \\ d S^n & -c \end{pmatrix} , \qquad
 \Gamma_1 = \begin{pmatrix} a & b^* \\ b & -a \end{pmatrix},
\]
and the operators $a$, $b$, $c$ and $d$ act as multiplication operators, $(a\psi)(x) = a(x) \psi(x)$.
A short computation shows that $\Gamma_0$ and $\Gamma_1$ are self-adjoint unitaries.
If $n=1$, then we can consider~$c$ and $d$ as constants in $\R$ and $\C$, respectively.
More generally, we can consider any product of self-adjoint unitaries $\Gamma_0, \Gamma_1 \in M_2(L^\infty(\Z)\rtimes \Z)$,
where $L^\infty(\Z)\rtimes \Z$ is the $\ast$-algebra generated by~$L^\infty(\Z)$, the algebra of bounded functions $\Z\to \C$,
and the shift operator $S$. See~\cite[Section~6]{Cedzich22} for more examples.
\end{Example}

\subsection{The symmetry indices for chiral unitaries of Cedzich et al.}

\begin{table}
\begin{center}\renewcommand{\arraystretch}{1.2}
 \begin{tabular}{ @{ }l l l l@{ } }
\hline
 Symmetry index & Assumption & Reference & Equivalent index \\
\hline
$\operatorname{si}_\pm(U)$ & $ \pm 1 \notin \sigma_{\rm ess}(U)$ & Definition \ref{def:si_pm_def}, \cite[Section~2]{Cedzich18} & \\
$\operatorname{si}(U)$ & $\{\pm 1\} \cap \sigma_{\rm ess}(U) = \varnothing$ & Definition \ref{def:si_pm_def}, \cite[Section~2]{Cedzich18} & \\
$\operatorname{si}(Q)$ & $0 \notin \sigma_{\rm ess}(Q)$ & equation \eqref{eq:Cedzich_selfadj_index}, \cite[Section~2]{Cedzich18} & \\
$\Ind_{\Gamma_0}(U)$ & $\{\pm 1\} \cap \sigma_{\rm ess}(U) = \varnothing$ & Definition \ref{def:suzuki_index}, \cite[Section~2]{Suzuki19} & $\operatorname{si}(U)$, $\operatorname{si}(\operatorname{Im}(U))$ \\
$\Ind_\pm (\Gamma_0, U)$ & $ \pm 1 \notin \sigma_{\rm ess}(U)$ & Definition \ref{def:Tanaka_symm_index}, \cite[Section~2]{MST21} & $\operatorname{si}_\pm(U)$ \\
$\Ind(P_0, P_1)$ & $\|P_0 - P_1\|_{\calQ(\calH)} < 1$ & Definition \ref{def:Ind(p,q)_HSpace}, \cite[Section~3]{ASS} & $\operatorname{si}_-(U)$ \\
$\Ind(P_0, \one - P_1)$ & $\| P_0 + P_1 - \one \|_{\calQ(\calH)} < 1$ & Definition \ref{def:Ind(p,q)_HSpace}, \cite[Section~3]{ASS} & $\operatorname{si}_+(U)$ \\
\hline
 \end{tabular}
 \caption{Summary of the symmetry indices for chiral unitaries $U=\Gamma_0 \Gamma_1 = (2P_0-\one)(2P_1- \one)$,
 $\Gamma_0 U \Gamma_0 = U^*$, self-adjoint operators $Q$ with chiral symmetry, $\Gamma_0 Q \Gamma_0 = -Q$, and
 the index of a pair of projections $P_0$, $P_1$. All operators are acting on a fixed separable Hilbert space $\calH$. } \label{table:symmetry_index_summary}
 \end{center}
\end{table}

Let us now review the indices for chiral unitaries $U = \Gamma_0\Gamma_1$ on $\calH$ defined by
Cedzich et al.~\cite{Cedzich18}.
 While we will define other topological indices for chiral
unitaries on Hilbert spaces, we emphasise that all other indices are merely a rewriting of the symmetry indices
introduced in~\cite{Cedzich18}. In particular, all results in this section can be considered as results about the symmetry
indices of Cedzich et al.
A summary of the indices we consider for operators on Hilbert spaces is given in Table \ref{table:symmetry_index_summary}.

Recall that, because $\Gamma_0 U \Gamma_0 = U^*$, the spectrum of $U$ is symmetric about the
real axis. In particular, the points $\pm 1$ of the spectrum are of interest as they are the only points that need not be
symmetric under $\Gamma_0$. The indices defined in~\cite{Cedzich18} precisely measure this possible
spectral asymmetry.

The spectral asymmetry of a self-adjoint and invertible operator $S$ on a finite-dimensional Hilbert space is given by
the signature $\operatorname{Sig}(S)$, which is the number of positive eigenvalues subtracted by
the number of negative eigenvalues (counting multiplicity).
Therefore, if ${\pm 1 \notin\! \sigma_{\rm ess}(U)}$, the space $\Ker(U \mp \one)$ is finite dimensional and we can
consider $\operatorname{Sig}\bigl( \Gamma_0 |_{\Ker(U \mp \one)} \bigr)$ as a way to probe the possible asymmetry of
the spectrum of $U$ with respect to $\Gamma_0$.

\begin{Definition} \label{def:si_pm_def}
Let $U = \Gamma_0 \Gamma_1$ be a chiral unitary on $\calH$. If $\pm 1 \notin \sigma_{{\rm ess}}(U)$,
the symmetry index of $U$ is given by
\[
 \operatorname{si}_\pm(U) = \operatorname{Sig}\bigl( \Gamma_0 |_{\Ker(U \mp \one)} \bigr) = \Tr\bigl( \Gamma_0 |_{\Ker(U \mp \one)} \bigr).
\]
If both indices are defined, the total symmetry index is defined by the sum $\operatorname{si}(U) = \operatorname{si}_+(U) + \operatorname{si}_-(U)$.
\end{Definition}

In the definition of $\operatorname{si}_\pm(U)$, we have used that the signature of a self-adjoint and unitary operator on a finite-dimensional
Hilbert space can be computed via its trace.
We also remark that various other indices are defined in~\cite{Cedzich18}
for $U$ satisfying other (possibly antilinear) symmetry relations by studying the action of the symmetry operator on the
$\pm 1$ eigenspaces of $U$.

Following the convention of the symmetries of free-fermionic Hamiltonians,
we say that ${Q=Q^*}$ is chiral-symmetric as a self-adjoint operator with respect to the self-adjoint unitary~$\Gamma_0$ if~$Q$ anti-commutes with $\Gamma_0$. One can therefore consider
a symmetry index of self-adjoint Fredholm operators $Q$ that are chiral-symmetric with respect to $\Gamma_0$ by defining
\begin{equation} \label{eq:Cedzich_selfadj_index}
 \operatorname{si}(Q) := \operatorname{Sig}\bigl( \Gamma_0 |_{\Ker(Q)} \bigr) = \Tr\bigl( \Gamma_0 |_{\Ker(Q)} \bigr),
\end{equation}
That is, $\operatorname{si}(Q)$ measures the spectral
asymmetry of $\Ker(Q)$ with respect to the self-adjoint unitary~$\Gamma_0$.
It was observed in~\cite[Section 4]{Cedzich18} that the total symmetry index $\operatorname{si}(U)$ can
be expressed using the symmetry index of the chiral-symmetric and self-adjoint operator $\operatorname{Im}(U) = \frac{1}{2\rm{i}} (U - U^*)$,
$\operatorname{si}(U) = \operatorname{si}(\operatorname{Im}(U))$.
The finite dimensionality of $\Ker(\operatorname{Im}(U))$ is equivalent to the condition that $\{\pm 1\} \cap \sigma_{{\rm ess}}(U)=\varnothing$~\cite[Lemma 3.7]{Cedzich18b}.

\subsection{The symmetry indices of Suzuki et al.} \label{subsec:suzuki_index}

Let us now review alternative descriptions of the symmetry indices of Cedzich et al.\ that have appeared in the
mathematical physics literature. As previously emphasised, these topological indices are not new, but provide
a different presentation of the previously-defined indices that may be useful depending on the example under study.

Let $U = \Gamma_0 \Gamma_1$ be a chiral unitary on $\calH$ and suppose that both $+1$ and $-1$ are
not contained in the essential spectrum.
The index of self-adjoint operators from equation \eqref{eq:Cedzich_selfadj_index} was further studied by Suzuki~\cite{Suzuki19}, who in particular emphasised the role of
supersymmetry.\footnote{Our use of the term `supersymmetry' is to align our terminology with~\cite{Suzuki19}
and denotes the setting of a $\Z_2$-graded Hilbert space with operators that change the parity of vectors.
Such a setting need not be related to any symmetry between bosons and fermions.}
Namely, given~$U = \Gamma_0\Gamma_1$, the operator $Q = \operatorname{Im}(U)$ anti-commutes with the
self-adjoint unitary $\Gamma_0$.
Hence there is a off-diagonal representation of $Q$ with respect to the spectral decomposition
of $\Gamma_0$.
Let~$\calH_{\pm} = (\Gamma_0 \pm \one)\calH = \Ker( \Gamma_0 \mp \one)$. Then we may write
\[
 \calH = \calH_+ \oplus \calH_-, \qquad
 Q = \begin{pmatrix} 0 & Q_+^* \\ Q_+ & 0 \end{pmatrix}, \qquad
 Q_+\colon\ \calH_+ \to \calH_-.
\]

\begin{Definition}[{\cite[Section 4]{Cedzich18} and \cite{Suzuki19}}] \label{def:suzuki_index}
We say that the chiral unitary $U= \Gamma_0 \Gamma_1$ is of supersymmetric Fredholm type if $Q_+$ is a Fredholm operator. In such a case,
we define the supersymmetric index of $U$ as the Fredholm index,
\[
 \Ind_{\Gamma_0}(U) = \Index(Q_+) = \dim \Ker(Q_+) - \dim \Ker(Q_+^*) \in \Z.
\]
\end{Definition}

Note that if $\Ker(Q)$ is finite dimensional, then it is well known (see, for example,~\cite[p.~124]{BGVbook}) that
\[
 \Index(Q_+) = \Tr\bigl( \Gamma_0 |_{\Ker(Q)} \bigr).
\]
Therefore, we immediately have an equivalence
\begin{equation} \label{eq:total_symm_index_Fred_index}
 \operatorname{si}(U) = \Ind_{\Gamma_0}(U) = \operatorname{si}(\operatorname{Im}(U)).
\end{equation}

Much like the decomposition of $\operatorname{si}(U)$ into a sum of $\operatorname{si}_\pm (U)$, a decomposition of
$\Ind_{\Gamma_0}(U)$ into a sum of two indices was given by Matsuzawa et al.~\cite{MST21, MTW21},
which we now review.

We first let $P_0$ be the projection onto the $+1$ eigenspace of $\Gamma_0$, i.e., $\Gamma_0 = 2P_0 - \one$.
Then using the decomposition of $\calH$ into even and odd subspaces,
$\calH = (\one + \Gamma_0)\calH \oplus (\one - \Gamma_0)\calH = P_0 \calH \oplus (\one - P_0)\calH$, one
defines $R_1 = \tfrac{1}{2} P_0( U+ U^*) P_0$ and $R_2 = \tfrac{1}{2} (\one -P_0)( U+ U^*)(\one - P_0)$ as the restriction of
$\operatorname{Re}(Q) = \tfrac{1}{2}(U+U^*)$ to the even and odd subspaces.

\begin{Definition} \label{def:Tanaka_symm_index}
If $\Ker( R_1 \pm \one)$ and $\Ker(R_2 \pm \one)$ are finite dimensional, we define
\[
 \Ind_\pm (\Gamma_0, U) = \dim \Ker(R_1 \mp \one) - \dim \Ker( R_2 \mp \one).
\]
\end{Definition}

It is already shown in~\cite[Section 2]{MST21} that $\Ind_\pm (\Gamma_0, U)$ is well defined if
$\pm 1 \notin \sigma_{{\rm ess}}(U)$ as well as the relations
\begin{equation} \label{eq:pm_index_si_index_connection}
 \Ind_\pm (\Gamma_0, U) = \operatorname{si}_\pm(U), \qquad
 \Ind_+ (\Gamma_0, U) + \Ind_- (\Gamma_0, U) = \Ind_{\Gamma_0}(U).
\end{equation}

\begin{Example}
Let us return to Example \ref{ex:canonical_example} with $n=1$,
\[
 \Gamma_0 = \begin{pmatrix} c & \ol{d} S^* \\ d S & -c \end{pmatrix} , \qquad
 \Gamma_1 = \begin{pmatrix} a & b^* \\ b & -a \end{pmatrix}, \qquad U = \Gamma_0 \Gamma_1,
\]
where the operators $a$ and $b$ act as multiplication operators, $(a\psi)(x) = a(x) \psi(x)$, and
$(c,d) \in \R\times \C$ are scalars such that $c^2 + |d|^2 = 1$.
 The unitary $\Gamma_0$ does not act locally, as is often demanded of a chiral symmetry operator, but can
be diagonalised to the more standard form~$\left(\begin{smallmatrix} \one & \hphantom{-}0 \\ 0 & -\one \end{smallmatrix}\right)$ by the unitary transformation
in~\cite[Theorem 8]{SuzukiTanaka}.\footnote{Ensuring that
 the symmetry operators act locally is particularly important for studying the bulk-boundary correspondence, which we do not
 consider in this paper.}
The index $\operatorname{si}(U) = \Ind_{\Gamma_0}(U)$ has been explicitly computed in the case
that the functions $a$ and $b$ have limits at $\pm \infty$~\cite{SuzukiTanaka}. The refined indices $\operatorname{si}_\pm(U) =\Ind_\pm( \Gamma_0, U)$
are also computed for explicit models under similar assumptions in~\cite{MST21, MTW21}.
Briefly, the results in~\cite{MST21, MTW21, SuzukiTanaka} use the fact that the essential spectrum of $U$ is characterised
by~$U$ at the limit points $\pm \infty$, which allows for precise conditions to determine whether $\pm 1 \notin \sigma_{{\rm ess}}(U)$.
 Taking compact perturbations of $U$ to a simpler unitary then allow the authors to compute the symmetry indices.
In~\cite[Section 5.5]{Cedzich18}, split-step quantum walks are constructed from shift operators and
rotation matrices, where the
angle of rotation may vary as a function of $\Z$. By constraining the possible range of the angle of rotation, Cedzich et al.\
show that the essential spectrum of the split-step quantum walk is gapped at $\pm 1$.
\end{Example}

\subsection{Connection to the index of a pair of projections} \label{subsec:Ind(p,q)_HSpace}

Our next task to connect the indices $\operatorname{si}_\pm(U)= \Ind_\pm(\Gamma_0,U)$ for the chiral unitary
$U = \Gamma_0 \Gamma_1$ to the index of a pair of projections~\cite{ASS}.
Let us first recall the following result.

\begin{Lemma}[{\cite[Section~6, Lemma 4]{Kasparov80}}] \label{lemma:Kas_unit_equiv_projs}
Let $B$ be a unital $C^*$-algebra and $p_0,p_1 \in B$ projections such that $\|p_0-p_1\|_{B} < 1$. Then
there exists a unitary $v \in B$ such that $p_1 = vp_0v^*$.
\end{Lemma}

\begin{Definition} \label{def:Ind(p,q)_HSpace}
Let $P_0$ and $P_1$ be self-adjoint projections on $\calH$.
We say that $(P_0,P_1)$ is a~Fredholm pair if
$\|P_0 - P_1\|_{\calQ(\calH)} < 1$.

If $(P_0,P_1)$ is a Fredholm pair, then the operator $P_1|_{\Ran(P_0)}\colon \Ran(P_0) \to \Ran(P_1)$ is a Fredholm operator and we define
\[
 \Ind(P_0,P_1) = \Index\bigl( P_1|_{\Ran(P_0)} \bigr).
\]
\end{Definition}

See~\cite[Lemma 4.1]{BCPRSW} for a proof that $P_1|_{\Ran(P_0)}$ is Fredholm if and only if $\|P_0 - P_1\|_{\calQ(\calH)} < 1$.
The index of a pair of projections is a generalisation of the essential codimension of self-adjoint
projections.

\begin{Lemma}[{\cite[Lemma 4.10]{BCLR}}] \label{lem:Ind(p,q)_formula}
If $(P_0,P_1)$ is a Fredholm pair of projections, then
\[
 \Ind(P_0,P_1) = \dim \bigl( \Ran(P_0) \cap \Ker(P_1) \bigr) - \dim \bigl( \Ker(P_0) \cap \Ran(P_1) \bigr).
\]
\end{Lemma}

We list some basic properties of $\Ind(P_0,P_1)$.
\begin{Lemma}[{cf.~\cite[Section 3]{ASS} and \cite[Lemma 4.3]{BCPRSW}}] \label{lem:Ind(p,q)_properties}
Let $P_0, P_1 \in \calB(\calH)$ be self-adjoint projections.
\begin{itemize}\itemsep=0pt
 \item[$(i)$] If $(P_0, P_1)$ is a Fredholm pair, then $\Ind(P_0,P_1) = - \Ind(P_1,P_0) = - \Ind\bigl( (\one - P_0), (\one - P_1) \bigr)$.
 \item[$(ii)$] If $(P_0, \one - P_1)$ is a Fredholm pair, then $\Ind(P_0,(\one - P_1)) = \Ind(P_1, (\one - P_0) )$.
 \item[$(iii)$] If $P_0$, $P_1$, $P_2$ are projections such that $\|P_0 - P_1\|_{\calQ(\calH)} < \tfrac{1}{2}$ and $\|P_1 - P_2\|_{\calQ(\calH)} < \tfrac{1}{2}$, then
 $\Ind(P_0,P_1)+\Ind(P_1,P_2) = \Ind(P_0,P_2)$.
\end{itemize}
\end{Lemma}

Parts (i) and (ii) of Lemma \ref{lem:Ind(p,q)_properties} can be proved quite easily using Lemma \ref{lem:Ind(p,q)_formula}. Part (iii)
is less trivial.

Let us now use $\Ind(P_0, P_1)$ to study chiral unitaries $U = \Gamma_0 \Gamma_1 \in \calB(\calH)$.
Because $\Gamma_0$ and $\Gamma_1$ are self-adjoint unitaries, there are self-adjoint
projections $P_0, P_1 \in \calB(\calH)$ such that $\Gamma_j = 2P_j - \one$, $j\in\{0,1\}$.
We first show the relation of the index of a pair of projections to the total symmetry
index $\operatorname{si}(U)= \Ind_{\Gamma_0}(U)$.

\begin{Proposition}[{cf.~\cite[equation (1.6)]{Suzuki19}}] \label{prop:HSpace_SuzukiIndex_is_ProjectionIndex}
Let $U = \Gamma_0 \Gamma_1 \in \calB(\calH)$ be a chiral unitary with
$\Gamma_j = 2P_j - \one$, $j\in\{0,1\}$. Then $Q = \operatorname{Im}(U) = \frac{1}{2\rm{i}}(U-U^*)$ is Fredholm
if and only if $(P_0,P_1)$ and $(P_0, \one - P_1)$ are Fredholm pairs and
\[
 \Ind(P_0, P_1) + \Ind(P_0, (\one - P_1)) = \operatorname{si}(Q) = \operatorname{si}(U) = \Ind_{\Gamma_0}(U).
\]
\end{Proposition}
\begin{proof}
Given $U = (2P_0 -\one)(2P_1 - \one)$, a simple computation gives that
\[
 Q = \frac{1}{2\rm{i}}(U-U^*) = \frac{1}{2\rm{i}}(\Gamma_0 \Gamma_1 - \Gamma_1\Gamma_0) = \frac{2}{\rm i} (P_0 P_1 - P_1P_0).
\]
Hence we have
\begin{align*}
 Q_+ &= \frac{1}{2} (\one -\Gamma_0) Q \frac{1}{2} (\one + \Gamma_0) = (\one - P_0)Q P_0 \\
 &= -2{\rm i} (\one - P_0)(P_0P_1-P_1P_0) P_0 = 2{\rm i} (\one - P_0) P_1 P_0
\end{align*}
and similarly $Q_+^* = -2{\rm i} P_0 P_1(\one - P_0)$. Therefore, $\Ker(Q)$ is finite dimensional and
$Q_+$ is Fredholm if and only if $\Ker\bigl( (\one - P_0) P_1 P_0 \bigr)$ and $\Ker\bigl( P_0 P_1(\one - P_0) \bigr)$
are finite dimensional. Recalling that $Q_+\colon P_0\calH \to (\one - P_0)\calH$ and $Q_+^*\colon (\one - P_0)\calH \to P_0\calH$, we
 furthermore find that
\begin{align*}
 \Ker(Q_+) &= \Ker( (\one - P_0)P_1 P_0) = \Ker( (\one - P_0)P_1) \cap \Ran(P_0) \\
 &= \bigl(\Ker(P_1) \oplus \Ran(P_1) \bigr) \cap \Ran(P_0),
\end{align*}
and similarly
\begin{align*}
 \Ker(Q_+^*) &= \Ker( P_0 P_1(\one - P_0) ) = \Ker( P_0 P_1) \cap \Ran( \one - P_0) \\
 &= \bigl( \Ker(P_1) \oplus \Ran(P_1) \bigr) \cap \Ker(P_0).
\end{align*}
That is,
\begin{align*}
 \Index(Q_+) ={}& \dim \bigl( \Ran(P_0) \cap \Ker(P_1) \bigr) + \dim \bigl( \Ran(P_0) \cap \Ran(P_1) \bigr) \\
 & - \dim \bigl( \Ker(P_0) \cap \Ran(P_1) \bigr) - \dim \bigl( \Ker(P_0) \cap \Ker(P_1) \bigr).
\end{align*}
Using Lemma \ref{lem:Ind(p,q)_formula}, if $(P_0, P_1)$ and $(P_0, \one - P_1)$ are Fredholm pairs, then
\begin{align*}
 &\Ind(P_0, P_1) = \dim \bigl( \Ran(P_0) \cap \Ker(P_1) \bigr) - \dim \bigl( \Ker(P_0) \cap \Ran(P_1) \bigr), \\
 &\Ind(P_0, \one - P_1) = \dim \bigl( \Ran(P_0) \cap \Ran(P_1) \bigr) - \dim \bigl( \Ker(P_0) \cap \Ker(P_1) \bigr)
\end{align*}
and the index equality $\Ind_{\Gamma_0}((2P_0-\one)(2P_1-\one)) = \Ind(P_0, P_1) + \Ind(P_0, (\one - P_1))$ follows.
The connection to $\operatorname{si}(U)= \operatorname{si}(\operatorname{Im}(U))$ then also follows from equation \eqref{eq:total_symm_index_Fred_index}.
We note that $\Ker(Q)$ is finite dimensional if and only if the spaces
$\Ker(P_0) \cap \Ran(P_1)$, $\Ker(P_1) \cap \Ran(P_0),$ $\Ker(P_0) \cap \Ker(P_1)$ and $\Ran(P_0)\cap \Ran(P_1)$ are finite dimensional,
which is equivalent to $(P_0, P_1)$ and $(P_0, \one - P_1)$ being Fredholm pairs.
\end{proof}

We can similarly relate $\Ind(P_0, P_1)$ and $\Ind(P_0, (\one - P_1))$ to the refined indices
$\operatorname{si}_\pm(U)$ and $\Ind_\pm( \Gamma_0, U)$. We first prove some preparatory results.

\begin{Lemma} \label{lem:U_kernel_to_P_kernel}
Let $U = \Gamma_0 \Gamma_1 =(2P_0-\one)(2P_1 - \one)$ be a chiral unitary on $\calH$. If
$(P_0, P_1)$ is a~Fredholm pair, then
\[
 \Ker( U+ \one) = \Ker(\Gamma_0 + \Gamma_1) = \Ran(P_0)\cap \Ker(P_1) \oplus \Ker(P_0)\cap \Ran(P_1).
\]
If $(P_0, \one - P_1)$ is a Fredholm pair, then
\[
 \Ker(U - \one) = \Ker( \Gamma_0 - \Gamma_1) = \Ran(P_0) \cap \Ran(P_1) \oplus \Ker(P_0) \oplus \Ker(P_1).
\]
\end{Lemma}
\begin{proof}
If $(P_0,P_1)$ is a Fredholm pair of projections if and only if the subspace
$\Ran(P_0)\cap \Ker(P_1) \oplus \Ker(P_0)\cap \Ran(P_1)$ is finite dimensional. We next note that
\[
\Ker( \Gamma_0 + \Gamma_1) = \Ker( P_0 + P_1 - \one) = \Ran(P_0)\cap \Ker(P_1) \oplus \Ker(P_0)\cap \Ran(P_1),
\]
where the last equality comes from~{\cite[Lemma 4.10]{BCLR}} (which can be proven easily).
We then see that
\[
 \Ker( \Gamma_0 + \Gamma_1) = \Ker\bigl( (\Gamma_0 + \Gamma_1)\Gamma_1 \bigr) = \Ker( U + \one)
\]
as $\Gamma_1$ is unitary.
The case of $\Ker(U-\one)$ follows the same argument using that $(P_0, \one - P_1)$ is a~Fredholm pair. We omit the details.
\end{proof}

\begin{Lemma}[{\cite[Lemma 3.7]{Cedzich18b}}] \label{lem:essential_spectrum_condition}
Let $U =(2P_0-\one)(2P_1 - \one)$ be a chiral unitary on $\calH$. The projections $(P_0, P_1)$
$($respectively $(P_0, \one - P_1))$ are a Fredholm pair if and only if $-1 \notin \sigma_{\rm ess}(U)$
$($respectively $1 \notin \sigma_{\rm ess}(U))$. Consequently $\operatorname{si}(U) = \Ind_{\Gamma_0}(U)$
is well defined if and only if
$\{\pm 1\} \cap \sigma_{\rm ess}(U)= \varnothing$.
\end{Lemma}
\begin{proof}
The result immediately follows from Lemma \ref{lem:U_kernel_to_P_kernel} as $\pm 1 \notin \sigma_{\rm ess}(U)$
if and only if ${\Ker(U \mp \one)}$ is finite dimensional.
\end{proof}

\begin{Proposition}
Let $U = \Gamma_0 \Gamma_1 =(2P_0-\one)(2P_1 - \one)$ be a chiral unitary on $\calH$.
If
$(P_0, \one - P_1)$ is a Fredholm pair, then
\[
 \Ind(P_0, \one- P_1) = \Ind_+( \Gamma_0, U) = \operatorname{si}_+(U).
\]
If $(P_0, P_1)$ is a Fredholm pair, then
\[
 \Ind(P_0, P_1) = \Ind_-( \Gamma_0, U) = \operatorname{si}_-(U).
\]
\end{Proposition}
\begin{proof}
Recall the operators $R_1 = \tfrac{1}{2} P_0( U+ U^*) P_0$ and $R_2 = \tfrac{1}{2} (\one -P_0)( U+ U^*)(\one - P_0)$ that
appear in the definition of $\Ind_\pm( \Gamma_0, U)$.
Elementary computations give that
$R_1 = P_0(2P_1- \one)P_0$ and $R_2 = -(\one-P_0)(2P_1 - \one)(\one - P_0)$, therefore
\begin{align*}
 \Ker( R_1 - \one) &= \Ker\bigl( P_0( 2P_1 -\one - \one) P_0 \bigr) = \Ker\bigl(P_0( \one - P_1)P_0 \bigr) \\
 &= \Ran(P_0) \cap \Ker( \one - P_1) = \Ran(P_0) \cap \Ran( P_1)
\end{align*}
and similarly
\begin{align*}
 \Ker(R_2 - \one)&= \Ker\bigl( (\one - P_0)(-2P_1) (\one - P_0) \bigr) = \Ran(\one - P_0) \cap \Ker(P_1)\\
 &= \Ker(P_0) \cap \Ker(P_1).
\end{align*}
Hence,
\begin{align*}
\Ind_+( \Gamma_0, U) &= \dim \Ker(R_1 - \one) - \dim \Ker( R_2 - \one) \\
 &= \dim \bigl( \Ran(P_0) \cap \Ran( P_1) \bigr) - \dim \bigl( \Ker(P_0) \cap \Ker(P_1) \bigr)
 = \Ind(P_0, \one- P_1)
\end{align*}
and the equality with $\operatorname{si}_+(U)$ follows from \eqref{eq:pm_index_si_index_connection}. An analogous argument also shows that
$\operatorname{si}_-(U) = \Ind_- (\Gamma_0, U) = \Ind(P_0, P_1)$.
\end{proof}

We can also prove the connection between the index of a pair of projections and $\operatorname{si}_\pm(U)$ without
first showing equality with $\Ind_\pm(\Gamma_0, U)$. Namely, if $(P_0, P_1)$ is a Fredholm pair, then
 because $P_0$ is the $+1$ eigenspace projection of $\Gamma_0$ and
$\Ker( U+ \one) = \Ran(P_0)\cap \Ker(P_1) \oplus \Ker(P_0)\cap \Ran(P_1)$ from Lemma \ref{lem:U_kernel_to_P_kernel},
\begin{align*}
 \operatorname{si}_-(U) &= \Tr\bigl( \Gamma_0 |_{\Ker(U + \one)} \bigr) = \dim \big[ \Ran(P_0) \cap \Ker(U+\one) \big] - \dim \big[ \Ker(P_0) \cap \Ker(U+\one) \big] \\
 &= \dim \bigl( \Ran(P_0) \cap \Ker(P_1) \bigr) - \dim \bigl( \Ker(P_0) \cap \Ran(P_1) \bigr) = \Ind(P_0, P_1).
\end{align*}
The equality $\operatorname{si}_+(U) = \Ind(P_0, \one - P_1)$ follows the same argument.

\subsection{Further properties}

We now use properties of the index of a pair of projections to further study the symmetry indices of chiral unitaries.

\begin{Proposition}
Let $U = (2P_0- \one)(2P_1 - \one )$ be a chiral unitary on $\calH$.
If $(P_0, P_1)$ is a~Fredholm pair, then $\dim \Ker(U+ \one) \geq |\Ind(P_0,P_1)|$.
If $(P_0, \one - P_1)$ is a~Fredholm pair, then $\dim \Ker(U- \one) \geq |\Ind(P_0,\one - P_1)|$.
\end{Proposition}
\begin{proof}
Using that
\begin{align*}
 &\Ind(P_0, P_1) = \dim \bigl( \Ran(P_0) \cap \Ker(P_1) \bigr) - \dim \bigl( \Ker(P_0) \cap \Ran(P_1) \bigr), \\
 &\Ind(P_0, \one - P_1) = \dim \bigl( \Ran(P_0) \cap \Ran(P_1) \bigr) - \dim \bigl( \Ker(P_0) \cap \Ker(P_1) \bigr),
\end{align*}
 the result follows from Lemma \ref{lem:U_kernel_to_P_kernel}.
\end{proof}

The index $\Ind(P_0, P_1)$ is well defined if $\|P_0 - P_1 \|_{\calQ(\calH)} < 1$. In particular, we can define
$\Ind(P_0, P_1)$ if $P_0 - P_1$ or equivalently $\Gamma_0 - \Gamma_1$ is a compact operator.
If we are in the special case that $(\Gamma_0 \pm \Gamma_1)^{2m+1}$ is trace-class for some odd integer $2m+1$,
we can write another trace formula to compute $\Ind(P_0, P_1)$ and $\Ind(P_0, \one - P_1)$.

\begin{Proposition}[{\cite[Theorem 4.1]{ASS}}] \label{prop:Hspace_traceclass_formula}
If $(\Gamma_0 - \Gamma_1)^{2m+1}$ is trace-class, then
\[
 \operatorname{si}_-(U) = \Ind(P_0, P_1) = \Tr\bigl( (P_0 - P_1)^{2m+1} \bigr) = \frac{1}{2^{2m+1}} \Tr\bigl( (\Gamma_0 - \Gamma_1)^{2m+1} \bigr).
\]
If $(\Gamma_0 + \Gamma_1)^{2m+1}$ is trace-class, then
\[
 \operatorname{si}_+(U) = \Ind(P_0,\one - P_1) = \Tr\bigl( (P_0 + P_1- \one)^{2m+1} \bigr) = \frac{1}{2^{2m+1}} \Tr\bigl( (\Gamma_0 + \Gamma_1)^{2m+1} \bigr).
\]
\end{Proposition}

We can use Proposition \ref{prop:Hspace_traceclass_formula} to write many non-trivial examples of
$\Ind(P_0, P_1)$ or $\Ind(P_0,\allowbreak\one - P_1)$ by taking
trace-class perturbations of projections, for example.
More general trace formulas for $\Ind(P_0, P_1)$ can be found in~\cite[Section 3]{CP2}.

Let us briefly consider the case where $\calH$ is finite dimensional.

\begin{Proposition}[{cf.~\cite[Lemma 2.5]{Cedzich18}}] \label{prop:finite_index_is_sig}
Let $U = \Gamma_0 \Gamma_1$ be a chiral unitary on a finite-di\-men\-sion\-al
Hilbert space $\calH$. Then $\operatorname{si}(U) = \Ind_{\Gamma_0}(U)= \Tr(\Gamma_0) = \operatorname{Sig}(\Gamma_0)$.
\end{Proposition}
Proposition \ref{prop:finite_index_is_sig} follows from Propositions \ref{prop:HSpace_SuzukiIndex_is_ProjectionIndex}
and \ref{prop:Hspace_traceclass_formula} with $m=0$. We also give a short proof using properties of
the index of a pair of projections.
\begin{proof}
If $\calH$ is finite dimensional, then $\sigma_{\rm ess}(U) = \varnothing$ and all
projections are Fredholm pairs.
We use Lemma \ref{lem:Ind(p,q)_properties} and compute
\begin{align*}
 \Ind_{\Gamma_0}(U)&= \Ind(P_0, P_1) + \Ind(P_0, \one - P_1) \\
 &= \Ind(P_0, P_1) + \Ind(P_1, \one - P_0) \\
 &= \Ind(P_0, \one - P_0) \\
 &= \dim \bigl( \Ran(P_0) \cap \Ker(\one - P_0) \bigr) - \dim \bigl( \Ker(P_0) \cap \Ran(\one - P_0) \bigr) \\
 &= \dim \Ran(P_0) - \dim \Ker(P_0) = \Tr(\Gamma_0) = \operatorname{Sig}(\Gamma_0)
\end{align*}
as $\Ran(P_0)$ and $\Ker(P_0)$ are the $+1$ and $-1$ eigenspaces of $\Gamma_0$, respectively.
\end{proof}

\begin{Remark}
The $\Z_2$-graded nature of $\Ind_{\Gamma_0}(U)$ allows us to obtain non-trivial indices even in finite
dimensions, where the ungraded Fredholm index will be zero. Indeed, considering $\C^m \oplus \C^n$ as
a $\Z_2$-graded Hilbert space, the operator
\[
 \mathbf{0} = \begin{pmatrix} 0_{m,m} & 0_{n,m} \\ 0_{m,n} & 0_{n,n} \end{pmatrix}, \qquad
 0_{j,k}\colon\ \C^j \to \C^k
\]
is such that $Q_+ := 0_{m,n}$ is Fredholm with
\[
 \Index(Q_+) = \dim \Ker(0_{m,n}) - \dim \Ker(0_{n,m}) = m-n.
\]
\end{Remark}

Clearly the argument used in the proof of Proposition \ref{prop:finite_index_is_sig} fails
 in infinite dimensions. Indeed, if $\calH$ is infinite dimensional, then $(P_0, \one - P_0)$ can not be a
Fredholm pair as any self-adjoint unitary $\Gamma_0$ on $\calH$ can not be trace class.
This simple observation leads to the following non-obvious result.

\begin{Proposition} \label{prop:trace_no_go}
Suppose $\calH$ is infinite dimensional and $U = \Gamma_0 \Gamma_1= (2P_0-\one)(2P_1 - \one)$ is a~chiral unitary
on $\calH$. Then $\| \Gamma_0 - \Gamma_1 \|_{\calQ(\calH)} \geq 1$ or $\| \Gamma_0 + \Gamma_1 \|_{\calQ(\calH)} \geq 1$.
\end{Proposition}
\begin{proof}
We have that $\Gamma_0 - \Gamma_1 = 2(P_0-P_1)$ and $ \Gamma_0 + \Gamma_1 = 2(P_0+P_1- \one)$.
Now suppose the result is false. Then $\|P_0 - P_1\|_{\calQ(\calH)}$ and $\|P_0 + P_1 - \one \|_{\calQ(\calH)}$ are
strictly less than $\tfrac{1}{2}$. Hence we can estimate
\[
 \| P_0 - (\one - P_0) \|_{\calQ(\calH)} \leq \|P_0 - P_1\|_{\calQ(\calH)} + \|P_0 + P_1 - \one \|_{\calQ(\calH)} < 1
\]
and $(P_0, \one - P_0)$ is a Fredholm pair. But this is only possible if $\calH$ is finite dimensional,
a~contradiction of our assumptions.
\end{proof}

\begin{Corollary}
If $\calH$ is infinite dimensional and $U = \Gamma_0 \Gamma_1= (2P_0-\one)(2P_1 - \one)$ is a chiral unitary
with $\{\pm 1 \} \cap \sigma_{\rm ess}(U) = \varnothing$, then $1 \leq \| \Gamma_0 - \Gamma_1 \|_{\calQ(\calH)} <2$ or
$1 \leq \| \Gamma_0 + \Gamma_1 \|_{\calQ(\calH)} < 2$.
\end{Corollary}
\begin{proof}
The bound $1 \leq \| \Gamma_0 \pm \Gamma_1 \|_{\calQ(\calH)}$ is Proposition \ref{prop:trace_no_go}. If
$U$ is a chiral unitary with $\{\pm 1 \} \cap \sigma_{\rm ess}(U) = \varnothing$, then both $(P_0, P_1)$ and $(P_0, \one-P_1)$ are
Fredholm pairs, $\|P_0 - P_1\|_{\calQ(\calH)}<1$ and $\|P_0 + P_1 - \one \|_{\calQ(\calH)}<1$,
which then gives the strict upper bound on $\| \Gamma_0 \pm \Gamma_1 \|_{\calQ(\calH)}$.
\end{proof}

Proposition \ref{prop:trace_no_go} shows that a tracial formula for $\operatorname{si}(U) = \Ind_{\Gamma_0}(U)$
of the form considered in Proposition \ref{prop:Hspace_traceclass_formula} is impossible if $\calH$ is
 infinite dimensional as $(\Gamma_0 +\Gamma_1)^{2m+1}$ and $(\Gamma_0 - \Gamma_1)^{2n+1}$
cannot simultaneously be trace-class.

Finally, let us note some basic stability properties of the indices for chiral unitaries we have considered.

\begin{Proposition} \label{prop:Hspace_strong_path_invariance}
Let $[0,1] \ni t \mapsto U(t) = \Gamma_0(t) \Gamma_1(t)= (2P_0(t)-\one)(2P_1(t) - \one)$ be a strongly continuous
path of chiral unitaries on $\calH$.
\begin{itemize}\itemsep=0pt
 \item[$(i)$] If $-1 \notin \sigma_{{\rm ess}}(U(t))$ for all $t \in [0,1]$, then $\Ind(P_0(t), P_1(t))$ is constant.
 \item[$(ii)$] If $+1 \notin \sigma_{{\rm ess}}(U(t))$ for all $t \in [0,1]$, then $\Ind(P_0(t), \one -P_1(t))$ is constant.
\end{itemize}
\end{Proposition}

This result is proved in a more general context in Proposition \ref{prop:homotopy_of_chiral_unitaries}, where for
Hilbert spaces the inequality $\| \pm U(t) - \one \|_{\calQ(\calH)} < 1$ is equivalent to the condition that $\mp 1 \notin \sigma_{{\rm ess}}(U(t))$.
In the Hilbert space setting, the result is relatively intuitive and we present a heuristic argument. Suppose $-1 \notin \sigma_{{\rm ess}}(U(t))$
for all $t\in [0,1]$. Then there is a $\delta >0$ such that
 range of the spectral projection $\chi_{({\rm e}^{{\rm i}(\pi + \delta)}, {\rm e}^{{\rm i}(\pi-\delta)})}(U(t))$ will be finite dimensional for all $t \in [0,1]$.
 Therefore, any change in the dimension of $\Ker(U(t) + \one)$ comes only from
eigenvalues of finite multiplicity. However, because the spectrum of $U(t)$ is symmetric about the real axis,
eigenvalues of $U(t)$ outside of~$\pm 1$ will come in pairs with respect to the spectral decomposition of $\Gamma_0(t)$.
So any change of dimension between $\Ker( U(t) + \one)$ and $\Ker( U(t+ \epsilon) + \one)$ will always be even dimensional and such that
\[
 \Tr\bigl( \Gamma_0(t) |_{\Ker(U(t) + \one)} \bigr) = \Tr\bigl( \Gamma_0(t+\epsilon) |_{\Ker(U(t+\epsilon) + \one)} \bigr).
\]
Hence, $\operatorname{si}_-(U(t)) = \operatorname{si}_-(U(t+ \epsilon))$ and therefore the index of a pair of projections will be constant.

\begin{Corollary}
The total symmetry index $\operatorname{si}(U) = \Ind_{\Gamma_0}(U)$ is invariant under strongly continuous paths of chiral unitaries
$\{U(t)\}_{t\in[0,1]}$ such that $\{\pm 1\}\cap \sigma_{{\rm ess}}(U(t)) = \varnothing$ for all $t\in [0,1]$.
\end{Corollary}

Because $\operatorname{si}_-(U) = \Ind(P_0,P_1)$ and $\operatorname{si}_+(U) = \Ind(P_0, \one-P_1)$ can be written as a Fredholm index,
these quantities are locally constant in the norm-topology. We give a concrete realisation of this result below. A similar result can be
found in~\cite[Proposition 2.6]{Cedzich18}.

\begin{Proposition}
Let $U = \Gamma_0 \Gamma_1$ and $U' = \Gamma_0 \Gamma_1'$ be two chiral unitaries with the same chiral symmetry operator $\Gamma_0$
and $\pi\colon \calB(\calH) \to \calQ(\calH)$ the quotient map onto the Calkin algebra.
Suppose that $\| \pi(\Gamma_0) \pm \pi(\Gamma_1) \|_{\calQ(\calH)} < 1$ and $\| U - U'\|_{\calB(\calH)} < 1$. Then
$\operatorname{si}_\pm(U) = \operatorname{si}_\pm(U')$.
\end{Proposition}
We do not need to assume that $\operatorname{si}_\pm(U')$ is well defined as this is implied by the hypothesis of the proposition.
Because the statement uses norm-bounds bounds in both $\calB(\calH)$ and in $\calQ(\calH)$, we will explicitly
write $\| \pi(S) \|_{\calQ(\calH)}$ (unlike the rest of this section).
\begin{proof}
We consider the case that $\| \pi(\Gamma_0) + \pi(\Gamma_1)\|_{\calQ(\calH)} < 1$ and $\| U - U'\|_{\calB(\calH)} < 1$
(the case $\| \pi(\Gamma_0) - \pi(\Gamma_1) \|_{\calQ(\calH)} < 1$ is analogous). Writing
$\Gamma_0 = 2P_0-\one$, $\Gamma_1 = 2P_1 - \one$ and $\Gamma_1' = 2P_1'-\one$, our assumptions imply that
\[
 \big\| \pi(P_0) + \pi(P_1) - \one \big\|_{\calQ(\calH)} = \frac{1}{2} \big\| \pi(\Gamma_0) + \pi(\Gamma_1) \big\|_{\calQ(\calH)} < \frac{1}{2}
\]
and
\[
 1 > \big\| U - U' \big\|_{\calB(\calH)} = \big\| \Gamma_0 ( \Gamma_1 - \Gamma_1') \big\|_{\calB(\calH)} = 2 \big\| P_1 - P_1' \big\|_{\calB(\calH)},
\]
so $\|\pi(P_1) - \pi(P_1')\|_{\calQ(\calH)} \leq \| P_1 - P_1'\|_{\calB(\calH)} < \tfrac{1}{2}$. Therefore, $(P_0, \one - P_1)$ and $(P_1, P_1')$ are Fredholm pairs. Applying
Lemma \ref{lem:Ind(p,q)_properties}, $(\one - P_1, \one - P_1')$ and $(P_0,\one - P_1')$ are also Fredholm pairs with
\[
 \Ind(P_0, \one - P_1') = \Ind(P_0, \one - P_1) + \Ind(\one - P_1, \one - P_1') = \Ind(P_0, \one - P_1) - \Ind(P_1, P_1').
\]
Because $\| P_1 - P_1'\|_{\calB(\calH)} < 1$, by Lemma \ref{lemma:Kas_unit_equiv_projs}
there is a unitary $V \in \calU(\calH)$ such that
$VP_1 V^* = P_1'$. In particular, unitarily equivalent Fredholm pairs of projections are such that $\Ind(P_1, P_1') = 0$ \cite[Theorem 3.3]{ASS}.
Putting these results together,
\[
 \operatorname{si}_+(U) = \Ind(P_0, \one-P_1) = \Ind(P_0, \one - P_1') - \Ind(P_1, P_1')
 = \Ind(P_0, \one - P_1') = \operatorname{si}_+(U').\tag*{\qed}
\] \renewcommand{\qed}{}
\end{proof}

\subsection{Connection to the generator/Hamiltonian} \label{subsec:H_space_generator}

Quantum walks are often regarded as a discretisation of a time evolution, so it is natural to
consider the case $U_H = {\rm e}^{{\rm i} \pi H}$ for some self-adjoint operator $H$ on $\calH$.
Recalling the symmetries of free-fermionic Hamiltonians,
 $H$ is chiral-symmetric as a self-adjoint operator with respect to the self-adjoint unitary $\Gamma_0$
if
$\Gamma_0 \cdot \Dom(H) \subset \Dom(H)$ and $\Gamma_0 H \Gamma_0 = - H$.
For such chiral-symmetric~$H$, we note that
$\Gamma_0 U_H \Gamma_0 = U_H^*$ and the discrete time evolution is a chiral unitary.
As in Section~\ref{subsec:suzuki_index}, we can
decompose $H$ as an off-diagonal operator under the spectral decomposition of $\Gamma_0$.
If $H$ is Fredholm, we can consider the index problem
$\Index(H_+)$ with $H_+ = (\one - P_0) H P_0$ and $\Gamma_0 = 2P_0 - \one$.
We note that $\Index(H_+)$ is the same as the symmetry index $\operatorname{si}(H)$ of the self-adjoint
and chiral-symmetric operator $H$ introduced in equation \eqref{eq:Cedzich_selfadj_index} and
considered more generally in~\cite[Section 2.6]{Cedzich18}.

\begin{Proposition} \label{prop:Generator_index_Hspace}
Let $H \in \calB(\calH)$ be self-adjoint with $\Gamma_0 H \Gamma_0 = -H$, $\|H\|\leq 1$ and
$U_H = {\rm e}^{{\rm i}\pi H} = (2P_0-\one)(2P_1-\one)$. If $H$ is Fredholm, then
\[
 \operatorname{si}(H) = \Index(H_+) = \Ind(P_0, \one - P_1) = \operatorname{si}_+(U_H).
\]
\end{Proposition}
The condition $\|H\| \leq 1$ can always be guaranteed by considering $H\bigl(\one +H^2\bigr)^{-1/2}$ for bounded or unbounded $H$.
We will prove an analogous result in the more general context of chiral unitaries on Hilbert $C^*$-modules in
Section \ref{subsec:Module_generator_index}. For completeness, we give a short and independent proof.
\begin{proof}
If $H_+$ is Fredholm, then the index can be computed by the restricted trace of $\Gamma_0$,
$\Index(H_+) = \Tr\bigl( \Gamma_0 |_{\Ker(H)} \bigr)$. Because $\Ker(H)$ is finite dimensional and
$\|H\| \leq 1$, we have that $\Ker(H) = \Ker(U_H-\one)$ and so $(P_0, \one - P_0)$ is a Fredholm pair of
projections by Lemma \ref{lem:U_kernel_to_P_kernel}. We now use
that $\Ind(P_0, \one - P_1) = \operatorname{si}_+(U_H)$ and so
\[
\Ind(P_0, \one - P_1) = \Tr\bigl( \Gamma_0 |_{\Ker(U_H - \one)} \bigr) = \Tr\bigl( \Gamma_0 |_{\Ker(H)} \bigr) = \Index(H_+). \tag*{\qed}
\]\renewcommand{\qed}{}
\end{proof}

\begin{Remark}[a warning with homotopies]
We have previously investigated the topological stability of $\Z$-valued indices for
chiral unitaries $U = (2P_0-\one)(2P_1 - \one)$ via the index of a pair of projections.
Proposition \ref{prop:Generator_index_Hspace} shows
that in the case that $U$ comes from the discrete time evolution of a self-adjoint,
chiral-symmetric and bounded Fredholm operator $H$, then the index theoretic information of
$H$ can be recovered by considering $\Ind(P_0, \one - P_1)$.

A natural question to consider is whether the index of $H$ is also connected to $\Ind(P_0, P_1)$.
 This is not the case. Furthermore, some care is required as homotopies
of the generator $H$ that leave $\Index(H_+)$ invariant need not leave $\Ind(P_0, P_1)$ invariant.
As a simple example, let $H \in \calB(\calH)$ be Fredholm and consider the homotopy $[0,1]\ni t \mapsto H_t = (2\|H\|)^{-t} H$,
a continuous path from $H$ to $\frac{1}{2 \|H\|}H$. Clearly this path
is such that $\Ker(H_t)$ and $\Index((H_t)_+)$ is constant for all $t\in [0,1]$.
 However, $\|H_1\| = \tfrac{1}{2}$ so
by the spectral mapping theorem $U_{H_1} = {\rm e}^{{\rm i}\pi H_1} = {\rm e}^{\frac{{\rm i} \pi}{2\| H\|} H}$ will have a spectral gap at $-1$.
In particular, $\Ker(U_{H_1} + \one) = \{0\}$. Therefore, for $P_0'$ and $P_1'$ such that ${\rm e}^{{\rm i}\pi H_1} = (2P_0' - \one)(2P_1' - \one)$,
$\big| \Ind(P_0', P_1') \big| \leq \dim \Ker(U_{H_1} + \one) = 0$ and so
$\Ind(P_0', P_1') = 0$. If, on the other hand, $H=H_0$ has a unique eigenvalue at $+1$, then
$\operatorname{si}_-\bigl( {\rm e}^{{\rm i}\pi H_0} \bigr) = \Ind(P_0, P_1) = \pm 1$ with ${\rm e}^{{\rm i}\pi H_0} = (2P_0 - \one) (2P_1 - \one)$.
\end{Remark}

\subsection{Connection to the Cayley transform} \label{subsec:HSpace_Cayley}

Let us now study topological properties of the chiral unitary $U = \Gamma_0 \Gamma_1$ on $\calH$ via
the Cayley transform of $U$. Because we will study the Cayley transform more thoroughly in Section \ref{sec:Cstar_mod_index},
we only provide a brief summary here. Recall that
\[
 \calC(U) = {\rm i}(\one + U)(\one-U)^{-1}, \qquad \Dom(\calC(U)) = \Ran(\one-U)
\]
and $\calC(U)$ is self-adjoint on the Hilbert subspace $\ol{\Ran(\one-U)}$, see, for example, \cite[Section X.1]{ReedSimon2}.

Let us first note that
\[
\Gamma_0(\one-U) = (\one-U^*)\Gamma_0 = (U-\one)U^*\Gamma_0,
\]
so $\Gamma_0$ preserves $\Ran(\one-U) = \Dom(\calC(U))$ and furthermore
\begin{align*}
 \Gamma_0 \calC(U ) \Gamma_0 &= (\one+U^*)(\one-U^*)^{-1} = \bigl( (\one + U^*)U \bigr) \bigl( (\one - U^*)U \bigr)^{-1}\\
 &= (U+\one)(U-\one)^{-1} = -\calC(U ).
\end{align*}
Therefore, $\calC(U)$ anti-commutes with $\Gamma_0$ and is chiral-symmetric as a self-adjoint operator.
We define
$\calC(U)_+ = (\one - P_0) \calC(U) P_0$ with
$\Gamma_0 = 2P_0 - \one$.

\begin{Proposition} \label{Prop:Hspace_Cayley_P0P1}
If $- 1 \notin \sigma_{{\rm ess}}(U)$, then $\calC(U)$ is Fredholm and
\[
\operatorname{si}(\calC(U)) =\Index(\calC(U)_+ ) = \Ind(P_0, P_1) = \operatorname{si}_-(U).
\]
If $+1 \notin \sigma_{{\rm ess}}(U)$, then
$\calC(-U)$ is Fredholm and
\[
 \operatorname{si}(\calC(-U)) =\Index(\calC(-U)_+ ) = \Ind(P_0, \one - P_1) = \operatorname{si}_+(U).
\]
\end{Proposition}

\begin{proof}
If $- 1 \notin \sigma_{{\rm ess}}(U)$, then $\Ker(U+\one)$ is finite dimensional and $(P_0, P_1)$ is a Fredholm pair
by Lemma \ref{lem:essential_spectrum_condition}.
 Therefore, $\Ker(\calC(U)) = \Ker(U+\one) \cap \Ran(\one - U)$ is also finite dimensional
and so $\calC(U)$ is Fredholm. Using $\Ran(U - \one) =\Ran(P_0 - P_1)$ and Lemma \ref{lem:U_kernel_to_P_kernel},
\[
 \Ker(\calC(U)) = \bigl( \Ran(P_0)\cap \Ker(P_1) \oplus \Ker(P_0)\cap \Ran(P_1) \bigr) \cap \Ran(P_0 - P_1).
\]
If $x \in \Ran(P_0)\cap \Ker(P_1)$, then $(P_0 - P_1)x = P_0x = x$ and so $x\in \Ran(P_0-P_1)$.
Similarly, $\Ker(P_0)\cap \Ran(P_1) \subset \Ran(P_0 -P_1)$ and so
\[
 \Ker(\calC(U)) = \Ran(P_0)\cap \Ker(P_1) \oplus \Ker(P_0)\cap \Ran(P_1).
\]
Now, $\Index(\calC(U)_+ ) $ measures the spectral asymmetry of $\Gamma_0 = 2P_0 - \one$ on $\Ker( \calC(U) )$. Hence,
\begin{align*}
 \Index(\calC(U)_+ ) &= \dim \bigl( \Ran(P_0) \cap \Ker(P_1) \bigr) - \dim \bigl( \Ker(P_0) \cap \Ran(P_1) \bigr)
 = \Ind(P_0, P_1).
\end{align*}
The case of $\Ind(P_0, \one - P_1)$ and $\calC(-U)$ is entirely analogous.
\end{proof}

Hence, all the previous results concerning $\operatorname{si}_-(U)=\Ind(P_0, P_1)$ and $\operatorname{si}_+(U)=\Ind(P_0, \one - P_1)$ can be expressed
in terms of the Fredholm index of $\calC(U)_+$ and $\calC(-U)_+$.

\section[Resum'{e} on Hilbert C\^{}*-modules and index theory]{Resum\'{e} on Hilbert $\boldsymbol{C^*}$-modules and index theory} \label{sec:NCIndex_theory_prelims}

As a preliminary to considering chiral unitaries on Hilbert $C^*$-modules, we first review the basics of
Hilbert $C^*$-modules and noncommutative index theory.

\subsection[Hilbert C\^{}*-modules]{Hilbert $\boldsymbol{C^*}$-modules} \label{subsec:C*_modules}

We refer the reader to~\cite{Lance} for a more detailed introduction and proofs.
Throughout this section, $A$ is a $\sigma$-unital $C^*$-algebra. If $A$ is separable, this condition
is equivalent to the existence of a~countable approximate unit.

\begin{Definition}
A Hilbert $C^*$-module is a right $A$-module $E_A$ with a map
$(\cdot\mid\cdot)_A\colon E_A \times E_A \to A$ that is linear in the second variable such that
for all $e_1, e_2 \in E_A$ and $a \in A$,
\[
 	(e_1 \mid e_2\cdot a)_A = (e_1\mid e_2)_A a, \qquad
 	( e_1\mid e_2)_A = (e_2 \mid e_1)_A^*, \qquad
 	(e_1 \mid e_1)_A \geq 0
\]
and $(e_1 \mid e_1)_A = 0$ if and only if $e_1=0$. Furthermore, $E_A$ is complete with
respect to the norm $\|e \| = \big\| ( e\mid e)_A \big\|_A^{1/2}$, $e \in E_A$.
\end{Definition}

\begin{Examples}\quad
\begin{itemize}\itemsep=0pt
 \item[(i)] A complex Hilbert space $\calH$ is a Hilbert $C^*$-module over $\C$ with $(\cdot\mid \cdot)_\C$ given by the
 usual Hilbert space inner product.

 \item[(ii)] The algebra $A$ can be seen as a Hilbert $C^*$-module $A_A$ with the structure
 \[
 a_1 \cdot a_2 = a_1 a_2, \qquad (a_1 \mid a_2)_A = a_1^* a_2, \qquad a_1, a_2 \in A.
 \]

 \item[(iii)] The standard Hilbert $C^*$-module $\ell^2(\N, A)$ is defined as follows:
 \[
 \ell^2(\N, A) = \bigg\{ (a_n)_{n \in \N}, a_n \in A \colon\ \sum_{n \in \N} a_n^* a_n \text{ converges in }A \bigg\}
 \]
 and has $C^*$-module structure
 \[
 (a_n)_{n\in \N} \cdot b = (a_n b)_{n\in \N}, \quad a_n, b \in A, \qquad
 \bigl( (a_n)_{n \in \N} \mid (b_n)_{n\in \N} \bigr)_A = \sum_{n\in \N} a_n^* b_n.
 \]
\end{itemize}
\end{Examples}

A Hilbert $C^*$-module $E_A$ is full if the closure in the norm of $A$ of the span of $\{ (e_1 \mid e_2)_A$: $e_1, e_2 \in E_A \}$
recovers $A$.
Both $A_A$ and $\ell^2(\N, A)$ are full Hilbert $C^*$-modules.
Kasparov's stabilisation theorem implies that for any countably generated Hilbert $C^*$-module $E_A$, there is an isometric embedding
$E_A \to \ell^2(\N, A)$~\cite{KasparovStinespring}.

\begin{Remark}
While Hilbert $C^*$-modules share many similarities to Hilbert spaces, one important difference is that if $F_A$ is a closed submodule
of $E_A$ with
\[
 F^\perp_A = \big\{ e \in E_A \colon (e\mid f)_A = 0 \text{ for all } f \in F_A \big\},
\]
then it is not guaranteed that $E_A = F_A \oplus F^\perp_A$ in general.
\end{Remark}

Given a map $T\colon E_A \to E_A$, we say that $T$ is adjointable if there is a map $T^*\colon E_A \to E_A$ that acts as the adjoint
with respect to the $A$-valued inner product on $E_A$. Such maps are right $A$-linear and bounded in the operator norm.
The set of adjointable operators on $E_A$ is denoted~$\End_A(E)$ and is a $C^*$-algebra with respect to the operator norm.
In analogy to the case of Hilbert spaces, where the compact operators arise as a norm-closure of finite-rank operators,
for any $e_1, e_2 \in E_A$, we define the operator $\Theta_{e_1,e_2}(e_3) = e_1 \cdot (e_2\mid e_3)_A$, $e_3 \in E_A$.
One finds that~$\Theta_{e_1,e_2}^* = \Theta_{e_2,e_1}$, so $\Theta_{e_1,e_2}$ is adjointable, and the
compact operators $\mathbb{K}_A(E)$ on $E_A$ are defined such that
\[
 \mathbb{K}_A(E) = \ol{\operatorname{span}\big\{ \Theta_{e_1,e_2} \colon e_1 , e_2 \in E_A \big\} }.
\]
The compact operators are a $2$-sided ideal in $\End_A(E)$ and we denote the
$C^*$-algebra $\calQ_A(E) = \End_A(E)/ \mathbb{K}_A(E)$
with $\pi\colon \End_A(E) \to \calQ_A(E)$ the quotient map. Note that $\calQ_A(E)$ is a
generalisation of the Calkin algebra to the setting of Hilbert $C^*$-modules.

\begin{Examples}\quad
\begin{itemize}\itemsep=0pt
 \item[(i)] Considering a Hilbert space $\calH$ as a Hilbert $C^*$-module over $\C$, the compact and adjointable operators are
 precisely the compact and bounded operators on $\calH$, respectively.

 \item[(ii)] Considering $A$ as a Hilbert $C^*$-module over itself, $\mathbb{K}_A(A) = A$ as multiplication is dense in $A$.
 The adjointable operators are isomorphic to the multiplier algebra of $A$, $\Mult(A) \cong \End_A(A)$.

 \item[(iii)] Let $\calH$ be a separable Hilbert space and $\calH \otimes A$ the Hilbert $C^*$-module with structure
 \[
 (\psi_1 \otimes a_1)\cdot a_2 = \psi_1 \otimes a_1 a_2, \qquad
 ( \psi_1 \otimes a_1 \mid \psi_2 \otimes a_2 )_A = \langle \psi_1, \psi_2 \rangle_\calH a_1^* a_2
 \]
 for $\psi_1,\psi_2 \in \calH$, $a_1, a_2 \in A$. Note that this Hilbert $C^*$-module is isomorphic to the standard Hilbert $C^*$-module $\ell^2(\N, A)$.
 In this case, $\mathbb{K}_A(\calH \otimes A) = \calK(\calH) \otimes A$ and
 $\End_A( \calH \otimes A) \cong \Mult( \calK(\calH) \otimes A)$.
\end{itemize}
\end{Examples}

\subsection[Fredholm operators on Hilbert C\^{}*-modules and K-theory]{Fredholm operators on Hilbert $\boldsymbol{C^*}$-modules and $\boldsymbol{K}$-theory}

We fix a $\sigma$-unital $C^*$-algebra $A$ and countably generated
Hilbert $C^*$-module $E_A$. Our exposition closely follows~\cite[Section 2]{Wahl07}.

We will occasionally need to deal with unbounded operators on Hilbert $C^*$-modules, though we will only
consider densely defined, closed and right $A$-linear operators ${D\colon \Dom(D) \!\! \subset\! E_A\! \to \!\!E_A}$.
We say that $D$
is \emph{regular} if $D^*$ is densely defined and the operator $\one +D^*D\colon \Dom (D^*D) \to E_{A}$ has dense range.
Self-adjoint and regular operators admit a continuous functional calculus. We let
$\Dom(D)_A$ be the Hilbert $C^*$-module that comes from the completion of $\Dom(D)$ in the graph
norm of $D$.
If $D$ is regular, $\one +D^*D$ has a bounded and positive inverse in $\End_A(E)$ with
$\Ran\bigl( (\one +D^*D)^{-1} \bigr) \subset \Dom(D)$ and dense in $E_A$~\cite[Lemma 9.2]{Lance}.
 A short computation then shows that the
square root
$(\one +D^*D)^{-1/2}\colon E_A \to \Dom(D)_A$ is a unitary map.
The operator $F_D:= D(\one + D^*D)^{-1/2} \in \End_A(E)$ is called the bounded transform of $D$.

\begin{Definition}
We say that an operator $T \in \End_A(E)$ is Fredholm on $E_A$ if $\pi(T) \in \calQ_A(E)$ is invertible.
We say that a regular operator $D\colon \Dom(D) \to E_A$ is Fredholm on $E_A$ if
$F_D = D(\one + D^*D)^{-1/2} \in \End_A(E)$ is Fredholm on $E_A$.
\end{Definition}

If $D\colon \Dom(D) \to E_A$ is a regular operator, then $D$ is Fredholm on $E_A$ if and only if
$D$ is Fredholm as a bounded operator $\Dom(D)_A \to E_A$.

\begin{Proposition}[{\cite[Proposition 2.1, Corollary 2.2]{Wahl07}}] \label{prop:Fredholm_properties}
Let $D$ be a regular self-adjoint operator on $E_A$. The following statements are equivalent.
\begin{itemize}\itemsep=0pt
 \item[$(i)$] $D$ is Fredholm,
 \item[$(ii)$] There is an $\varepsilon >0$ such that for all $\varphi \in C_c(-\varepsilon, \varepsilon)$, $\varphi(D) \in \mathbb{K}_A(E)$.
 \item[$(iii)$] There is an $\varepsilon \!>\!0$ such that for any continuous function $\chi\colon\R \to \R$ such that ${\chi |_{(-\infty, -\varepsilon]} = -1}$
 and $\chi |_{[\varepsilon, \infty)} = +1$, $\chi(D)^2 - \one \in \mathbb{K}_A(E)$.
\end{itemize}
\end{Proposition}

\begin{Definition} \label{def:Reg_fn}
Let $D$ be self-adjoint and regular on $E_A$.
We say that a smooth odd non-decreasing function $\chi\colon \R \to \R$ is a normalising function for
$D$ if $\chi'(0)>0$, $\lim_{x\to \infty}\chi(x) = 1$ and $\chi(D)^2 - \one \in \mathbb{K}_A(E)$.
\end{Definition}

We note that if $D$ is self-adjoint, regular and $D$ has compact resolvent, $(i+D)^{-1} \in \mathbb{K}_A(E)$, then
the bounded transform
$x \mapsto x\bigl(1+x^2\bigr)^{1/2}$ is a normalising function for $D$.

If $S \in \End_A(E)$, we will (again)
abuse notation and write $\| S \|_{\calQ_A(E)}$ to mean $\| \pi(S) \|_{\calQ_A(E)}$ with
$\pi\colon \End_A(E) \to \calQ_A(E)$ the quotient map.

\begin{Lemma}[{\cite[Lemma 2.7]{Wahl07}}] \label{lem:unbdd_Fred_condition}\quad
\begin{itemize}\itemsep=0pt
 \item[$(i)$] A self-adjoint element $T \in \End_A(E)$ with $\| T\| \leq 1$ is Fredholm if and only if \linebreak ${\big\| \one - T^2 \big\|_{\calQ_A(E)} < 1}$.
 \item[$(ii)$] A self-adjoint regular operator $T$ is Fredholm if and only if $\big\| ({\rm i}+T)^{-1} \big\|_{\calQ_A(E)} < 1$.
\end{itemize}
\end{Lemma}
\begin{proof}
The inequality $\big\| \one - T^2 \big\|_{\calQ_A(E)} < c < 1$ holds if and only if $\pi(T)^2$ is bounded from below by $\one-c$.
Taking a
continuous function $\chi\colon\R \to \R$ such that $\chi |_{(-\infty, -(1-c)]} = -1$,
 $\chi |_{[1-c, \infty)} = +1$, we see that $\chi(T)^2 - \one \in \mathbb{K}_A(E)$.
Part (i) then follows by Proposition \ref{prop:Fredholm_properties}.

For (ii), we write $F_T = T\bigl( \one + T^2\bigr)^{-1/2} = T( (T+{\rm i})^*(T+{\rm i}) )^{-1/2}$ and see that
\begin{align*}
 \one - F_T^2 &= \one - T^2\bigl( (T+{\rm i})^*(T+{\rm i}) \bigr)^{-1}
 = \bigl( (T+{\rm i})^*(T-{\rm i}) - T^2\bigr) \bigl( (T+{\rm i})^*(T+{\rm i}) \bigr)^{-1} \\
 &= \bigl( (T+{\rm i})^*(T+{\rm i}) \bigr)^{-1}.
\end{align*}
Therefore, $\big\| \one - \ F_T^2 \big\|_{\calQ_A(E)} = \big\| (T+{\rm i})^{-1} \big\|_{\calQ_A(E)}^2$ and by part (i)
$\pi( F_T)$ is invertible if and only if $ \big\| (T+{\rm i})^{-1} \big\|_{\calQ_A(E)} < 1$.
\end{proof}

Self-adjoint Fredholm operators on Hilbert $C^*$-modules yield topological information of
$C^*$-algebras via Kasparov's
$KK$-theory~\cite{Kasparov80}, which we briefly introduce in a simpler context.
If there is a $\Z_2$-grading $E_A \cong \bigl(E^0 \oplus E^1\bigr)_A$, we say that a densely
defined operator
$D$ is even (respectively odd) if $D\cdot \bigl(E^i \cap \Dom(D)\bigr) \subset E^i$
(respectively $D \cdot \bigl(E^i \cap \Dom(D)\bigr) \subset E^{i+1}$) for $i \in \Z_2$.

\begin{Definition}
Let $E_A$ be a countably generated Hilbert $C^*$-module and $F \in \End_A(E)$ a~self-adjoint operator such that
$\one - F^2 \in \mathbb{K}_A(E)$.
If there is a $\Z_2$-grading $E_A \cong \bigl(E^0 \oplus E^1\bigr)_A$ such that~$F$ acts as an odd operator, we
 call the triple $(\C, E_A, F)$ an even Kasparov module. Otherwise we call $(\C, E_A, F)$ an odd Kasparov
 module.
\end{Definition}

Two Kasparov modules $\bigl(\C, E^{(0)}_A, F_0\bigr)$ and $\bigl(\C, E^{(1)}_A, F_1\bigr)$ are unitarily equivalent if there is an
even unitary $U\colon E^{(1)}_A \to E^{(2)}$
such that $UF_1 U^* = F_2$.
We say that even/odd Kasparov modules are homotopic if there is an even/odd Kasparov module
$\bigl(\C, \tilde{E}_{A \otimes C([0,1])}, F\bigr)$ such that the evaluating the fibre at $0$ and $1$ yields Kasparov modules
that are unitarily equivalent to $\bigl(\C, E^{(0)}_A, F_0\bigr)$ and $\bigl(\C, E^{(1)}_A, F_1\bigr)$, respectively.

Homotopy equivalence classes of even and odd Kasparov modules yields the abelian groups $KK(\C, A)$ and $KK^1(\C, A)$, respectively,
where the group operation is by direct sum~\cite[Section~4]{Kasparov80}.

\begin{Remarks} \label{Remark_KKgroup_notes}\quad
\begin{itemize}\itemsep=0pt
 \item[(1)] If $F \in \End_A(E)$ is a self-adjoint Fredholm operator such that $F^2 = \one$, then the Kasparov module
 $\bigl( \C, E_A, F \bigr)$ is called degenerate. Degenerate Kasparov modules represent the group identity in
 $KK(\C, A)$ or $KK^1(\C, A)$.
 \item[(2)] Suppose that $\bigl(\C, \bigl(E^0 \oplus E^1\bigr)_A, F\bigr)$ is an even Kasparov module. We can ignore the $\Z_2$-grading
 of $E_A$ and instead consider this triple as an odd Kasparov module. However, this odd Kasparov module will be trivial
 in $KK^1(\C, A)$. To see this, let $\gamma = \gamma^* = \gamma^{-1}$ denote the grading operator of $E_A$
 \big(that is, $\gamma e^i = (-1)^i e^i$ for $e^i \in E^i$ and $i \in \Z_2$\big). Then
 by assumption $F \gamma + \gamma F = 0$ as $F$ is odd. One can then check that the path
 $[0,1] \ni t \mapsto F_t = \cos(\tfrac{\pi}{2} t) F + \sin(\tfrac{\pi}{2} t) \gamma \in \End_A(E)$ is a homotopy of Fredholm operators on~$E_A$ and
 we can define the odd Kasparov module
\[
 \bigl( \C, C([0,1], E)_{A\otimes C([0,1])}, F_\bullet \bigr), \qquad (F_\bullet e)(t) = F_t e(t).
\]
Evaluating at the point $t=1$, the Kasparov module $(\C, E_A, \gamma)$ is a degenerate odd Kasparov module and
so represents the identity in $KK^1(\C, A)$.
\end{itemize}
\end{Remarks}

\begin{Proposition}[{\cite[Proposition 2.14]{vdDungen17}}] \label{prop:Fredholm_to_KK}
Let $D$ be a regular, self-adjoint and Fredholm operator on $E_A$ and $\chi$ a
normalising function for $D$. Then the triple $\bigl(\C, E_A, \chi(D) \bigr)$ is an odd
Kasparov module. If there is a $\Z_2$-grading $\bigl(E^0 \oplus E^1\bigr)_A$ such that $D$ is an
odd operator, then the triple $\bigl(\C, E_A, \chi(D)\bigr)$ is an even Kasparov module.
The equivalence class of this Kasparov module in $KK(\C, A)$ or $KK^1(\C, A)$ is independent
of the choice of normalising function.
\end{Proposition}

\begin{Definition}
Given a self-adjoint, regular and Fredholm operator $D$ on $E_A$, we denote by~$[D] \in KK(\C, A)$ or $KK^1(\C, A)$ the equivalence class of the Kasparov
module from Proposition~\ref{prop:Fredholm_to_KK}.
\end{Definition}

\begin{Remark}
 If $F_0, F_1 \in \End_A(E)$ are self-adjoint Fredholm operators such that $F_0 - F_1 \in \mathbb{K}_A(E)$, then
$[F_0] = [F_1] \in KK(\C, A)$ or $KK^1(\C, A)$, which can be shown by the simple homotopy
$[0,1] \ni t \mapsto F_t = F_0 + t(F_1 - F_0)$.
\end{Remark}

\begin{Lemma} \label{prop:anticommuting_Fredholms_homotopic}
Let $A$ be a $C^*$-algebra and $E_A$ a countably generated Hilbert $C^*$-module.
Suppose that $T_0$ and $T_1$ are self-adjoint, regular and Fredholm on $E_A$,
$\Dom(T_0 T_1) = \Dom(T_1 T_0)$ is dense and $T_1 T_0 + T_0 T_1 = 0$ on this domain.
Then $[T_0] = [T_1] \in KK^1(\C, A)$.
If $E_A$ is $\Z_2$-graded and~$T_0$ and $T_1$ are odd operators, then $[T_0]=[T_1] \in KK(\C,A)$.
\end{Lemma}
\begin{proof}
We first fix normalising functions $\chi_0$ and $\chi_1$ for $T_0$ and $T_1$, respectively.
Then $\chi_0(T_0)$ and $\chi_1(T_1) \in \End_A(E)$ also anti-commute. We can now
argue analogously to part 2 of Remarks~\ref{Remark_KKgroup_notes}, where
\[
 \bigl( \C, C([0,1], E)_{A\otimes C([0,1])}, F_\bullet \bigr), \qquad
 (F_\bullet e)(t) = \bigl(\cos\big(\tfrac{\pi}{2} t\big) \chi_0(T_0) + \sin\big(\tfrac{\pi}{2} t\big) \chi_1(T_1) \bigr) e(t)
\]
defines a Kasparov module and
so gives a homotopy in $KK(\C, A)$ or $KK^1(\C,A)$.
\end{proof}

There are isomorphisms of $KK$-theory to operator algebraic $K$-theory for trivially graded $C^*$-algebras,
$KK(\C, A) \xrightarrow{\simeq} K_0(A)$ and $KK^1(\C, A) \xrightarrow{\simeq} K_1(A)$,
see~\cite[Section~6]{Kasparov80} or~\cite[Section~3]{BKR2}. We will further examine the
isomorphism $KK(\C, A) \cong K_0(A)$ in Section \ref{subsec:Index_pair_projs}.

\begin{Remark}
Given separable $C^*$-algebras $A$ and $B$, the more general group $KK(A, B)$ is constructed
from equivalence classes of triples
$\bigl(A, {}_\phi E_B, F \bigr)$, where $E_B$ is a $\Z_2$-graded Hilbert $C^*$-module,
$F$ is an odd self-adjoint operator and there is a representation $\phi\colon A \to \End_B(E)$
such that $\phi(a)\bigl( \one - F^2\bigr)$ and $[ \phi(a), F]_{\pm}$ are compact for all $a \in A$ with
$[ \cdot, \cdot]_\pm$ the $\Z_2$-graded commutator.
The notation ${}_\phi E_B$ is used to denote both the (right) Hilbert $C^*$-module $E_B$ and the (left)
representation $\phi\colon A \to \End_B(E)$.
We note that
$KK^1(\C, A) \cong KK(\Cl_1, A)$ with $\Cl_1$ the complex Clifford algebra with one generator.
While we will mostly work with $KK(\C, A)$ and $KK^1(\C, A)$, we will occasionally take advantage
of the existence of the Kasparov product, a map
\[
 KK(A, B) \times KK(B, C) \to KK(A, C)
\]
for separable $C^*$-algebras $A$, $B$ and $C$. The Kasparov product also equips
$KK(A, A)$ with the structure of a ring.
\end{Remark}

\subsection{The index of a pair of projections} \label{subsec:Index_pair_projs}

In this section we review an analogue of the index of a pair of projections in the Hilbert $C^*$-module setting.
Most of the content can be extracted from~\cite[Section~6]{Kasparov80}. See also~\cite[Section~3.2]{Wahl07}.
We again fix a $\sigma$-unital $C^*$-algebra $A$ and countably generated Hilbert $C^*$-module~$E_A$.

\begin{Definition}
We say that the projections $p_0,p_1 \in \End_A(E)$ are a Fredholm pair if
$\| p_0 - p_1 \|_{\calQ_A(E)} < 1$.
\end{Definition}

Recalling Lemma \ref{lemma:Kas_unit_equiv_projs}, any Fredholm pair of projections $p_0, p_1 \in \End_A(E)$
will be unitarily equivalent in the quotient algebra $\calQ_A(E)$.
Constructing a Kasparov module from a Fredholm pair of projections is well known,
 but for completeness we recall the details.

\begin{Proposition} \label{prop:Module_index_of_projections}
Let $(p_0,p_1)$ be a Fredholm pair of projections. Then for any
lift $V$ of a~unitary $v \in \calQ_A(E)$ such that $v\pi( p_0)v^* = \pi(p_1)$, the triple
\[
 \biggl( \C, p_0 E_A \oplus p_1 E_A, \begin{pmatrix} 0 & p_0 V^* p_1 \\ p_1 V p_0 & 0 \end{pmatrix} \biggr)
\]
is a Kasparov module. If $V'$ is another lift of $v$, then the two Kasparov modules represent the same class in $KK(\C, A)$.
\end{Proposition}
\begin{proof}
Using that $Vp_0V^* = p_1$ modulo compacts, we check that
\[
 \begin{pmatrix} 0 & p_0 V^* p_1 \\ p_1 Vp_0 & 0 \end{pmatrix}^2 - \one =
 \begin{pmatrix} p_0 V^* p_1 Vp_0 - p_0 & 0 \\ 0 & p_1 V p_0 V^* p_1 - p_1 \end{pmatrix} = 0 \quad {\rm mod} \ \mathbb{K}_A(E).
\]
Finally, any other lift $V'$ of $v$ is such that $V- V' \in \mathbb{K}_A(E)$. Therefore
$\left(\begin{smallmatrix} 0 & p_0 V'^* p_1 \\ p_1 V' p_0 & 0 \end{smallmatrix}\right)$ is also a~compact perturbation of
$\left(\begin{smallmatrix} 0 & p_0 V^* p_1 \\ p_1 V p_0 & 0 \end{smallmatrix}\right)$. Hence their classes coincide in $KK(\C, A)$.
\end{proof}

\begin{Definition}
Let $(p_0,p_1)$ be a Fredholm pair of projections in $\End_A(E)$. We denote by
$[(p_0,p_1)] \in KK(\C, A)$ the equivalence class of the Kasparov module from Proposition \ref{prop:Module_index_of_projections}.
\end{Definition}

The element $[(p_0,p_1)]$ can be considered as a $K$-theoretic analogue of the index of a pair of projections in $\calB(\calH)$.
Indeed, when $A = \calK(\calH)$, the isomorphism $KK(\C, \calK(\calH)) \xrightarrow{\simeq} \Z$ maps
$[(p_0,p_1)]$ to $\Ind(p_0,p_1)$ introduced in Section \ref{subsec:Ind(p,q)_HSpace}.

\begin{Remark}
Let $(p_0,p_1)$ be a Fredholm pair of projections in $\End_A(E)$. Using the same proof as
Proposition \ref{prop:Module_index_of_projections}, we can also construct the Kasparov module
\[
 \bigg( \C, p_0 \mathbb{K}_A(E) \oplus p_1 \mathbb{K}_A(E), \begin{pmatrix} 0 & p_0 V^* p_1 \\ p_1 V p_0 & 0 \end{pmatrix} \bigg),
 \qquad \pi( V p_0 V^* ) = \pi(p_1),
\]
where $\pi(V)\in \calQ_A(E)$ is unitary and $\mathbb{K}_A(E)$ is the Hilbert $C^*$-module over itself.
We therefore obtain
an element $\big[ (p_0, p_1)_{\mathbb{K}_A(E)} \big] \in KK(\C, \mathbb{K}_A(E))$.
When $E_A$ is a full Hilbert $C^*$-module, $\mathbb{K}_A(E)$ is Morita equivalent to $A$, which
gives an isomorphism $KK(\C, \mathbb{K}_A(E)) \xrightarrow{\simeq} KK(\C, A)$ that maps
$\big[ (p_0, p_1)_{\mathbb{K}_A(E)} \big]$ to $[(p_0, p_1)]$.

We also note that any class in $KK(\C, A)$ can be represented by an even Kasparov module
$\bigl( \C, \tilde{E}_A, F \bigr)$ such that $\tilde{E}_A$ is a full Hilbert $C^*$-module.
\end{Remark}

If $A$ is a unital $C^*$-algebra with ideal $J$ and quotient $q\colon A \to A/J$, then the relative $K$-theory group $K_0(A, J)$ can be
constructed via projections $p_0, p_1 \in A$ such that $q(p_0)$ and $q(p_1)$ are
unitarily equivalent in $A/J$, see, for example, \cite[Section~5.4]{Blackadar}. There is also
an excision isomorphism $K_0( A, J) \cong K_0( J)$.
Therefore, if $(p_0,p_1)$ is a Fredholm pair of projections in~$\End_A(E)$, then by Lemma
\ref{lemma:Kas_unit_equiv_projs} we obtain a relative $K$-theory class
\[
 [p_0] - [p_1] \in K_0( \End_A(E), \mathbb{K}_A(E)).
\]

\begin{Theorem}[{\cite[Section~6, Theorem 3]{Kasparov80}}] \label{Theorem:KK_to_K_Kasparov}
Let $(p_0,p_1)$ be a Fredholm pair of projections in~$\End_A(E)$. Then the map
$K_0( \End_A(E), \mathbb{K}_A(E)) \to KK(\C, \mathbb{K}_A(E))$ given by
\[
 [p_0] - [p_1] \mapsto
 \bigg[ \bigg( \C, p_0 \mathbb{K}_A(E) \oplus p_1 \mathbb{K}_A(E),
 \begin{pmatrix} 0 & p_0 V^* p_1 \\ p_1 V p_0 & 0 \end{pmatrix} \bigg) \bigg],
 \qquad \pi( V p_0 V^* ) = \pi(p_1),
\]
is well defined and an isomorphism of groups.
\end{Theorem}

We now consider some additional properties of $[(p_0,p_1)]$ that help justify our terminology as
the index of a pair of projections.

\begin{Lemma} \label{lem:[(p,q)]_properties}\quad
\begin{itemize}\itemsep=0pt
 \item[$(i)$] If $(p_0, p_1)$ is a Fredholm pair of projections in $\End_A(E)$, then $[(p_0, p_1)]^{-1} = [(p_1,p_0)] \in KK(\C, A)$.
 \item[$(ii)$] If $(p_0, \one - p_1)$ is a Fredholm pair of projections in $\End_A(E)$, so is $(p_1, \one -p_0)$
 and $[(p_0, \one - p_1)] = [(p_1, \one - p_0)] \in KK(\C, A)$.
\end{itemize}
\end{Lemma}
\begin{proof}
For part (i) we will show that $[(p_0,p_1)] \oplus [(p_1, p_0)]$ is the group identity in $KK(\C, A)$. Taking a unitary
$v \in \calQ_A(E)$ such that $v \pi(p_0) v^* = \pi(p_1)$ with lift $V\in \End_A(E)$, we have the sum of Kasparov modules
\begin{align*}
 &\biggl( \C, p_0 E_A \oplus p_1 E_A, \begin{pmatrix} 0 & p_0 V^* p_1 \\ p_1 V p_0 & 0 \end{pmatrix} \biggr) \oplus
 \biggl( \C, p_1 E_A \oplus p_0 E_A, \begin{pmatrix} 0 & p_1 V p_0 \\ p_0 V^* p_1 & 0 \end{pmatrix} \biggr) \\
 &\quad = \biggl( \C, \bigl( p_0 E_A \oplus p_1 E_A \bigr) \oplus \bigl( p_1 E_A \oplus p_0 E_A, \bigr),
 \begin{pmatrix} 0 & p_0 V^* p_1 \\ p_1 V p_0 & 0 \end{pmatrix} \oplus \begin{pmatrix} 0 & p_1 V p_0 \\ p_0 V^* p_1 & 0 \end{pmatrix} \biggr).
\end{align*}
We decompose the direct sum of Hilbert $C^*$-modules into its even and odd parts and write the sum as the Kasparov module
\[
 \left( \C, \bigl( p_0 E_A \oplus p_1 E_A \bigr) \oplus \bigl( p_0 E_A \oplus p_1 E_A, \bigr),
 \begin{pmatrix} 0 & 0 & 0 & p_0V^*p_1 \\ 0 & 0 & p_1V p_0 & 0 \\ 0 & p_0V^*p_1 & 0 & 0 \\ p_1 V p_0 & 0 & 0 & 0 \end{pmatrix} \right).
\]
The Fredholm operator of this Kasparov module anti-commutes with the operator $\sigma_1 \otimes \sigma_3$
with $\sigma_1 = \left(\begin{smallmatrix} 0 & \one \\ \one & 0 \end{smallmatrix}\right)$ and $\sigma_3 = \left(\begin{smallmatrix} \one & \hphantom{-} 0 \\ 0 & -\one \end{smallmatrix}\right)$.
Therefore, by Proposition \ref{prop:anticommuting_Fredholms_homotopic} the sum
$[(p_0, p_1)] \oplus [(p_1, p_0)]$ can be represented by the Kasparov module
\[
 ( \C, \bigl( p_0 E_A \oplus p_1 E_A \bigr) \oplus \bigl( p_0 E_A \oplus p_1 E_A, \bigr), \sigma_1 \otimes \sigma_3 ),
\]
which is degenerate and so represents the group identity in $KK(\C, A)$.

For part (ii), it is immediate that $(p_0, \one - p_1)$ is a Fredholm pair if and only if
$(p_1, \one - p_0)$ is a Fredholm pair. We also note that if $Vp_0 V^* = \one - p_1$ modulo $\mathbb{K}_A(E)$, then
$V(\one - p_0)V^* = p_1$ modulo $\mathbb{K}_A(E)$.
We will show that $[(p_0, \one -p_1)] \oplus [(p_1, \one - p_0)]^{-1} = [(p_0, \one -p_1)] \oplus [(\one - p_0, p_1)]$ is the
identity in $KK(\C, A)$. The sum can be represented by the Kasparov module
\begin{align*}
 &\left( \C, \bigl( p_0 E_A \oplus (\one - p_0)E_A \bigr) \oplus \bigl( p_1 E_A \oplus (\one -p_1) E_A \bigr),
 \vphantom{\begin{pmatrix} 0 & 0 & 0 & p_0 V^* (\one - p_1) \\ 0 & 0 & (\one - p_0)V^* p_1 & 0 \\ 0 & p_1 V (\one- p_0) & 0 & 0 \\ (\one - p_1)V p_0 \end{pmatrix}}
 \right. \\
 &\left.\qquad \begin{pmatrix} 0 & 0 & 0 & p_0 V^* (\one - p_1) \\ 0 & 0 & (\one - p_0)V^* p_1 & 0 \\ 0 & p_1 V (\one- p_0) & 0 & 0 \\ (\one - p_1)V p_0 \end{pmatrix} \right).
\end{align*}
We make the identifications
\[
 p_0 E_A \oplus (\one - p_0)E_A \cong E_A, \qquad p_1 E_A \oplus (\one -p_1) E_A \cong E_A,
\]
which simplifies the Kasparov module to
\[
 \biggl( \C, E_A \oplus E_A, \begin{pmatrix} 0 & V^* \\ V & 0 \end{pmatrix} \biggr).
\]
But this Kasparov module represents the index $[(\one, \one)]\in KK(\C, A)$ as indeed $\pi(V \one V^*) = \pi(\one)$ in $\calQ_A(E)$.
Hence by part (i), the sum $[(p_0, \one -p_1)] \oplus [(\one - p_0, p_1)]$ is the group identity.
\end{proof}

Lastly, we prove an addition formula for Fredholm pairs of projections using the group structure in $KK(\C, A)$.

\begin{Proposition} \label{prop:Index_pair_projection_sum_Kasmod}
Suppose that $(p_0, p_1)$ and $(p_1,p_2)$ are Fredholm pairs of projections in $\End_A(E)$
with $v, w \in \calQ_A(E)$ unitaries such that $v \pi(p_0) v^* = \pi(p_1)$ and $w \pi(p_1) w^* = \pi(p_2)$.
Then
\begin{equation} \label{eq:p0_p2_mod}
\biggl( \C, p_0 E_A \oplus p_2 E_A, \begin{pmatrix} 0 & V^* p_1 W^* \\ W p_1 V & 0 \end{pmatrix} \biggr),
 \qquad \pi(V) = v, \qquad \pi(W) = w,
\end{equation}
is a Kasparov module whose equivalence class $[(p_0,p_2)] = [(p_0,p_1)] \oplus [(p_1,p_2)] \in KK(\C, A)$.
\end{Proposition}

According to our definition, $(p_0, p_2)$ need not be a Fredholm pair of projections. However, this is not required in order to build
a Kasparov module and class $[(p_0,p_2)] \in KK(\C, A)$.

\begin{proof}
Given the unitaries $v, w \in \calQ_A(E)$,
it's then easy to see that $wv \in \calQ_{A}(E)$ is also unitary and
$wv \pi(p_0) v^* w^* = \pi(p_2)$.
Therefore, we can build the Kasparov module
\[
 \biggl( \C, p_0 E_A \oplus p_2 E_A, \begin{pmatrix} 0 & p_0 V^* W^* p_2 \\ p_2 WV p_0 & 0 \end{pmatrix} \biggr).
\]
Taking compact perturbations, this Kasparov module is equivalent in $KK(\C, A)$ to
\begin{equation*}
 \biggl( \C, p_0 E_A \oplus p_2 E_A, \begin{pmatrix} 0 & V^* p_1 W^* \\ W p_1 V & 0 \end{pmatrix} \biggr).
\end{equation*}

We now take the sum of the two Kasparov modules
\begin{align*} \label{eq:sum_module}
 & \biggl( \C, p_0 E_A \oplus p_1 E_A, \begin{pmatrix} 0 & p_0 V^* p_1 \\ p_1 V p_0 & 0 \end{pmatrix} \biggr)
 \oplus \biggl( \C, p_1 E_A \oplus p_2 E_A, \begin{pmatrix} 0 & p_1 W^* p_2 \\ p_2 W p_1 & 0 \end{pmatrix} \biggr)
 \nonumber \\
 &\quad= \biggl( \C, \bigl(p_0 E_A \oplus p_1 E_A \bigr) \oplus \bigl( p_1 E_A \oplus p_2 E_A \bigr),
 \begin{pmatrix} 0 & p_0 V^* p_1 \\ p_1 V p_0 & 0 \end{pmatrix} \oplus \begin{pmatrix} 0 & p_1 W^* p_2 \\ p_2 W p_1 & 0 \end{pmatrix} \biggr).
\end{align*}

On the summed Kasparov module, we first rearrange the direct sum
\[
 \bigl(p_0 E_A \oplus p_1 E_A \bigr) \oplus \bigl( p_1 E_A \oplus p_2 E_A \bigr) \to
 \bigl(p_0 E_A \oplus p_2 E_A \bigr) \oplus \bigl( p_1 E_A \oplus p_1 E_A \bigr).
\]
Applying this rearranging, we split the Kasparov module into a sum,
\[
 \biggl( \C, p_0 E_A \oplus p_2 E_A, \begin{pmatrix} 0 & V^* p_1 W^* \\ W p_1 V & 0 \end{pmatrix} \biggr)
 \oplus \biggl( \C, p_1 E_A \oplus p_1 E_A, \begin{pmatrix} 0 & V p_0 V^* \\ W^* p_2 W & 0 \end{pmatrix} \biggr),
\]
where we have freely taken compact perturbations of the operators (which will not change the $KK$-class).
Finally, $V p_0 V^*$ and $W^* p_2 W$ equal $p_1$ modulo $\mathbb{K}_A(E)$. Hence the second term in the sum is equivalent
to the Kasparov module
\[
 \biggl( \C, p_1 E_A \oplus p_1 E_A, \begin{pmatrix} 0 & \one \\ \one & 0 \end{pmatrix} \biggr),
\]
which is degenerate and represents the group identity in $KK(\C, A)$. Therefore
\begin{align*}
 &\biggl( \C, p_0 E_A \oplus p_2 E_A, \begin{pmatrix} 0 & V^* p_1 W^* \\ W p_1 V & 0 \end{pmatrix} \biggr)
 \oplus \biggl( \C, p_1 E_A \oplus p_1 E_A, \begin{pmatrix} 0 & V p_0 V^* \\ W^* p_2 W & 0 \end{pmatrix} \biggr) \\
 &\qquad\sim \biggl( \C, p_0 E_A \oplus p_2 E_A, \begin{pmatrix} 0 & V^* p_1 W^* \\ W p_1 V & 0 \end{pmatrix} \biggr),
\end{align*}
which is the same Kasparov module as equation \eqref{eq:p0_p2_mod}.
\end{proof}

\subsection[The map (tau\_B)\_*: KK(C,B) to R]{The map $\boldsymbol{(\tau_B)_\ast\colon KK(\C, B) \to \R}$} \label{subsec:induced_trace}

Finally, we recall the induced map on $KK(\C, B)$ from a faithful norm-lower semicontinuous
trace $\tau_B$ on a $C^*$-algebra $B$. In the context of this paper, we only consider the case where
$B$ is unital and $\tau_B(\one) = 1$. Further details and
proofs can be found in~\cite[Section 4]{ElementsNCG}.

We first recall the induced trace on operator algebraic $K$-theory. An element
$[p_0] - [p_1] \in K_0(B)$ is represented by a pair of projections $p_0, p_1 \in M_N(B)$
for some sufficiently large $N$. Writing these matrices as
$p_\bullet = \bigl( (p_\bullet)_{j,k} \bigr)_{j,k=1}^N$, $(p_\bullet)_{jk} \in B$, $\bullet \in \{0,1\}$,
we can then easily define a~map
\[
 (\tau_B)_\ast([p_0] - [p_1] ) = (\Tr \otimes \tau_B)( p_0 - p_1) = \sum_{j=1}^N ( \tau_B( (p_0)_{jj} ) - \tau_B( (p_1)_{jj} ) ) \in \R
\]
that respects the equivalence relation on $K_0(B)$.

Our task is to translate this map to the setting of Hilbert $C^*$-modules and Fredholm operators.

\begin{Definition}
We say that a Hilbert $C^*$-module $E_B$ is finitely generated and projective if there is a
set $\{y_j\}_{j=1}^n \subset E_B$ such that $\one = \sum_{j=1}^n \Theta_{y_j, y_j}$.
\end{Definition}

If $E_B$ is finitely generated and projective, we call $\{y_j\}_{j=1}^n$ a finite frame for
$E_B$. We note that for finitely generated and projective Hilbert $C^*$-modules,
$\one$ is a compact endomorphism and so~$\End_B(E) = \mathbb{K}_B(E)$. In particular,
the triple $\bigl( \C, E_B, 0 \bigr)$ is a Kasparov module.

Using the finite frame $\{y_j\}_{j=1}^n$,
we can define a projection
\[
 p = \bigl( p_{jk} \bigr)_{j,k = 1}^n \in M_n(B), \qquad p_{jk} = ( y_j \mid y_k)_B,
\]
and there are maps
\begin{alignat*}{5}
 &S\colon\ &&E_B \to (B_B)^n,\qquad &&R\colon\ &&(B_B)^n \to E_B,& \\
 &&&S(e) = \bigl( ( y_j\mid e)_B \bigr)_{j=1}^n,
 \qquad&&&&R \bigl( (b_j)_{j=1}^n \bigr) = \sum_{j=1}^n y_j\cdot b_j,&
\end{alignat*}
that restrict to isomorphisms $S\colon E_B\to p(B_B)^n$ and $R\colon p(B_B)^n\to E_B$.
We now note the following crucial result.

\begin{Proposition}[{\cite[Section 4.2]{ElementsNCG}}]
Let $B$ be a unital $C^*$-algebra, $E_B = (E^0 \oplus E^1)_B$ a $\Z_2$-graded Hilbert
$C^*$-module and $T = \left(\begin{smallmatrix} 0 & F^* \\ F & 0 \end{smallmatrix}\right) \in \End_B(E)$
a self-adjoint Fredholm operator. If $F\colon E^0_B \to E^1_B$ has closed range, then the Hilbert
$C^*$-submodules $\Ker(F)$ and $\Ker(F^*)$ are finitely generated and projective.
\end{Proposition}

The condition that $F$ has closed range is equivalent to the existence of a Moore--Penrose inverse~\cite[Theorem 2.2]{XuSheng},
an adjointable operator $G:E^1_B \to E^0_B$ such that
\[
 GFG = G, \qquad FGF = F, \qquad (FG)^* = FG, \qquad (GF)^* = GF.
\]

Once we obtain finitely generated and projective Hilbert $C^*$-modules $\Ker(F)$ and $\Ker(F^*)$, we
can construct projections $p_0 \in M_n(B)$ and $p_1 \in M_m(B)$ from the finite frames
and take the induced trace. We can extend this to a map on $KK(\C, B)$ via the following
result.

\begin{Proposition}[{\cite[Sections 4.3--4.4]{ElementsNCG}}]
Let $T \in \End_B(E)$ be an odd self-adjoint Fredholm operator on the $\Z_2$-graded
Hilbert $C^*$-module $E_B$. Then $[T] = \big[\tilde{T}\big] \in KK(\C, B)$ with
$\tilde{T} = \left(\begin{smallmatrix} 0 & \tilde{F}^* \\ \tilde{F} & 0 \end{smallmatrix}\right) \in \End_B\bigl(\tilde{E}\bigr)$
an odd self-adjoint Fredholm operator on $\tilde{E}_B$ such that
$\tilde{F}$ has closed range. Furthermore, $[T]$ can be represented by the
even Kasparov module
\[
 \bigl( \C, \Ker\bigl(\tilde{F}\bigr)_B \oplus \Ker\bigl( \tilde{F}^*\bigr)_B, 0 \bigr),
\]
with grading operator $\left(\begin{smallmatrix} \one & \hphantom{-}0 \\ 0 & -\one \end{smallmatrix}\right)$.
\end{Proposition}

We will write the map $(\tau_B)_\ast \colon KK(\C, B) \to \R$ as
\begin{align*}
 (\tau_B)_\ast ([ T])={}& \tau_B\bigl( \Ker\bigl(\tilde{F}\bigr) \bigr) - \tau_B \bigl( \Ker\bigl( \tilde{F}^*\bigr) \bigr) \\
:={}& (\Tr \otimes \tau_B )\big( p_{\Ker(\tilde{F})} \big) - (\Tr \otimes \tau_B )\big( p_{\Ker(\tilde{F}^*)} \big),
\end{align*}
where $p_{\Ker(\tilde{F})}$ and $p_{\Ker(\tilde{F}^*)}$ are the projections constructed from the
finite frames on $\Ker\bigl(\tilde{F}\bigr)$ and $\Ker\bigl( \tilde{F}^*\bigr)$, respectively.

\section[Chiral unitaries on Hilbert C\^{}*-modules and K-theory]{Chiral unitaries on Hilbert $\boldsymbol{C^*}$-modules and $\boldsymbol{K}$-theory} \label{sec:Cstar_mod_index}
In this section, we extend our definition of chiral unitary to adjointable operators acting on
Hilbert $C^*$-modules. As Hilbert spaces are an example of Hilbert $C^*$-modules, the contents of
Section \ref{sec:HSpace_chiral_unitaries}
can be recovered as a special case.

\subsection{The Cayley transform for chiral unitaries} \label{subsec:Module_Cayley}

We will freely use definitions and results from Section \ref{sec:NCIndex_theory_prelims}.
Throughout this section, we will fix a complex $\sigma$-unital $C^*$-algebra $A$ and a
countably generated Hilbert $C^*$-module $E_A$.
The adjointable and
compact operators on this Hilbert $C^*$-module will be denoted by~$\End_A(E)$ and~$\mathbb{K}_A(E)$
respectively. We also denote the quotient algebra $\calQ_A(E) = \End_A(E)/ \mathbb{K}_A(E)$.

\begin{Definition}
We say that a unitary operator $u \in \End_A(E)$ is a chiral unitary with
respect to a self-adjoint unitary $\gamma_0 \in \End_A(E)$ if $\gamma_0 u \gamma_0 = u^*$.
\end{Definition}

Applying Lemma \ref{lem:chiral_iff_decomp}, $u \in \End_A(E)$ is a chiral unitary with
respect to a self-adjoint uni\-tary~$\gamma_0 \in \End_A(E)$ if and only if there is a
self-adjoint unitary $\gamma_1 \in \End_A(E)$ such that $u = \gamma_0 \gamma_1$.

The Cayley transform of chiral unitaries considered in Section \ref{subsec:HSpace_Cayley} can
also be employed in the Hilbert $C^*$-module setting.

\begin{Definition}
Let $u \in \End_A(E)$ be a chiral unitary with respect to $\gamma_0 \in \End_A(E)$.
The Cayley transform of $u$ is the operator
\[
 \calC(u) = {\rm i}(\one + u)(\one - u)^{-1} , \qquad \Dom(\calC(u)) = (\one - u)E_A.
\]
\end{Definition}

\begin{Lemma}[{\cite[Chapter 10]{Lance}}] \label{lem:chiral_cayley_selfadj}
Let $u \in \End_A(E)$ be a chiral unitary with respect to $\gamma_0 \in \End_A\!(E)$. Then
$\calC(u)$ is a self-adjoint and regular operator on $\ol{(\one\! -\! u)E}_{\!A}$, the closure of $\Dom(\calC(u)\!)$ in the
module norm of $E_A$. Furthermore, $\calC(u)$ anti-commutes with $\gamma_0$.
\end{Lemma}
\begin{proof}
It is proved in~\cite[Chapter 10]{Lance} that $\calC(u )$ is self-adjoint and regular.
We next note that~$\gamma_0$ preserves the domain of $\calC(u )$,
\[
 \gamma_0(\one - u)E_A = (\one - u^*)\gamma_0 E_A = (u-\one)u^* \gamma_0 E_A \subset (\one - u)E_A
\]
and furthermore
\begin{align*}
 \gamma_0 \calC(u ) \gamma_0 &= (\one + u^*)(\one - u^*)^{-1} = (u + \one)(u - \one)^{-1} = -\calC(u ). \tag*{\qed}
\end{align*}\renewcommand{\qed}{}
\end{proof}

\begin{Definition} \label{Def:u_Fred_type}
We say that a chiral unitary $u \in \End_A(E)$ with respect to $\gamma_0 \in \End_A(E)$ is
of Fredholm type if $\|\one - u \|_{\calQ_A(E)} < 2$.
\end{Definition}

Our terminology is justified by the following result.

\begin{Proposition}
If $u \in \End_A(E)$ is a chiral unitary with respect to $\gamma_0 \in \End_A(E)$ of Fredholm type,
then $\calC(u)$ is Fredholm on $\ol{(\one - u)E}_A$, the closure of $\Dom(\calC(u))$ in the
module norm of $E_A$.
\end{Proposition}
\begin{proof}
Recalling Lemma \ref{lem:unbdd_Fred_condition}, it suffices to show that $\big\| ({\rm i} + \calC(u))^{-1} \big\|_{\calQ_A(\ol{(\one - u)E})} < 1$.
We can then easily compute that on $\Dom(\calC(u))$
\[
 {\rm i} + \calC(u) = {\rm i}\bigl( (\one - u) + (\one + u) \bigr)(\one - u)^{-1} = 2{\rm i}(\one - u)^{-1},
\]
hence $({\rm i}+\calC(u))^{-1} = -\tfrac{{\rm i}}{2}(\one - u)$ and extends to a map $E_A \to \Dom(\calC(u))$. We therefore have that
\[
 \big\| ({\rm i} + \calC(u))^{-1} \big\|_{\calQ_A(\ol{(\one - u)E})} = \frac{1}{2} \| u- \one \|_{\calQ_A(E)} < 1. \tag*{\qed}
\]\renewcommand{\qed}{}
\end{proof}

The self-adjoint unitary $\gamma_0$ gives a $\Z_2$-grading of the Hilbert $C^*$-module
$\ol{(\one - u)E}_A$ with $\calC(u)$ an odd operator. Therefore, if $u \in \End_A(E)$ is a chiral
unitary with respect to
$\gamma_0 \in \End_A(E)$ of Fredholm type, we obtain an element
$\big[ \calC(u) \big] \in KK(\C, A) \cong K_0(A)$ by Proposition \ref{prop:Fredholm_to_KK}.

\begin{Remark}
Suppose that $u \in \End_A(E)$ is a \emph{self-adjoint} chiral unitary,
$u = u^* = \gamma_0 u \gamma_0$. Hence $u$ commutes with $\gamma_0$. Using
that $u^{-1} = u$, one can show that $\calC(u ) = -\calC(u )$ and so is the zero map.
Similarly the Hilbert $C^*$-module $\ol{(\one - u)E}_A = pE_A$ with $p= \tfrac{1}{2}(\one - u)$. Hence, if
$\|u - 1 \|_{\calQ_A(E)} < 2$, then the class $\big[ \calC(u) \big]$ is determined by
the $-1$ eigenspace projection $ \tfrac{1}{2}(\one - u)$.
\end{Remark}

For completeness, let us also study the bounded transform
$F_{\calC(u)} = \calC(u)\bigl( \one + \calC(u)^2\bigr)^{-1/2}$ on $\ol{(\one - u)E}_A$.
We first compute on the appropriate domain using the normality of $u$,
\begin{align*}
 \one + \calC(u)^* \calC(u ) &= \one + \gamma_0 (\one - u^*)^{-1}(\one + u^*)(\one + u)(\one - u)^{-1}\gamma_0 \\
 &= \one + (2+ u+u^*)(2-u-u^*)^{-1} \\
 &= \bigl( (2-u-u^*)+(2+u+u^*) \bigr)(2-u-u^*)^{-1} \\
 &= 4(2-u-u^*)^{-1} = 4\bigl((\one - u)^* (\one - u)\bigr)^{-1}.
\end{align*}
Therefore, $\bigl( \one + \calC(u)^2 \bigr)^{-1/2} = \tfrac{1}{2} (2- u - u^*)^{1/2} = \tfrac{1}{2}|\one - u|$.
We define $V$ to be the (unitary) completion of the operator
$(\one - u)e \mapsto |\one - u|e$ on the dense subspace $(\one - u)E_A$ (cf.~\cite[Proposition 3.8]{Lance}).
Then we find that
\begin{align} \label{eq:C(u)_bdd_transfrom}
 F_{\calC(u)} &= \frac{{\rm i}}{2} (\one + u)(\one - u)^{-1} (2-u-u^*)^{1/2} = \frac{{\rm i}}{2} (\one + u)(\one - u)^{-1} |\one - u|
 = \frac{{\rm i}}{2} (\one + u) V
\end{align}
and
\begin{align*}
 F_{\calC(u)}^* F_{\calC(u)} &= \frac{1}{4} |\one - u| (\one - u^*)^{-1}(\one + u^*)(\one + u) (\one - u) |\one - u| \\
 &= \frac{1}{4} (\one + u^*)(\one + u) = \frac{1}{4}(2+u+u^*),
\end{align*}
which then implies
\[
 \one- F_{\calC(u)}^2 = \frac{1}{2} - \frac{1}{4}(u+u^*) = \frac{1}{4}(2-u-u^*) = \frac{1}{4} (\one - u)^* (\one - u).
\]

If $\one - u \in \mathbb{K}_A(E)$, then $\one - F_{\calC(u)}^2 \in \mathbb{K}_A(\ol{(\one - u)E})$
and $\calC(u)$ has compact resolvent, $({\rm i} + \calC(u))^{-1} \in \mathbb{K}_A(\ol{(\one - u)E})$.

We now prove a key stability result of our $K$-theoretic index for chiral unitaries.

\begin{Proposition} \label{prop:homotopy_of_chiral_unitaries}
Let $\{u(t)\}_{t \in [0,1]} \subset \End_A(E)$ and $\{\gamma_0(t)\}_{t \in [0,1]} \subset \End_A(E)$ be
strongly continuous paths of unitaries
such that $u(t)$ is a chiral unitary with respect to $\gamma_0(t)$ for all $t\in [0,1]$.
If $\|u(t) - \one\|_{\calQ_A(E)} < 2$ for all $t \in [0,1]$, then
 $\big[ \calC(u(t))\big]$ is constant in $KK(\C,A)$.
\end{Proposition}
\begin{proof}
Our assumptions are such that for all $t \in [0,1]$, $\calC(u(t))$ is an odd, self-adjoint
and Fredholm operator on $\ol{(\one - u(t))E}$ and so has a normalising function $\chi_t$.
Let us therefore fix a~function $\chi$ such that $\chi$ is a normalising function of
 $\calC( u(t) )$ for all $t \in [0,1]$. (This can be done via
Proposition \ref{prop:Fredholm_properties} and taking a function $\chi$ such that
$\operatorname{supp}(\chi) \subset \operatorname{supp}(\chi_t)$ for all $t\in [0,1]$.)
Then by~\cite[Lemma~1.1]{Wahl08},
the path
$[0,1] \ni t \mapsto \chi( \calC(u(t)) ) \in \End_A( \ol{(\one - u(t))E} )$ is strongly continuous.
In particular, the function
$[0,1] \ni t \mapsto \one - \chi( \calC(u(t)) )^2 \in \mathbb{K}_A( \ol{(\one - u(t))E} )$ and is \emph{norm} continuous.
Therefore, we can construct an even Kasparov module
\[
 \Big( \C, C( [0,1], \ol{(\one - u(\cdot) )E})_{C[0,1] \otimes A}, \chi\big[ \calC(u(\cdot)) \big] \Big),
 \qquad \big[ \chi( \calC(u(\cdot)) )g\big](t) = \chi\big[ \calC(u(t))\big] g(t),
\]
where $g \in C( [0,1], \ol{(\one - u(\cdot) )E})_{C[0,1] \otimes A}$ is such that $g(t) \in \ol{(\one - u(t))E}_A$.
We therefore have obtained a homotopy of even
Kasparov modules in $KK(\C, A)$ and so $\big[ \calC(u(t))\big]$ is constant.
\end{proof}

\begin{Corollary}
If $u_0$ is homotopic to $\one$ in $\End_A(E)$ via a path $\{u_t\}_{t \in [0,1]}$ of chiral unitaries
with respect to $\gamma_0$ of Fredholm type,
then $\big[ \calC(u_0) \big]$ is
trivial in $KK(\C, A)$.
\end{Corollary}

\begin{Remark}
If $u \in \End_A(E)$ is a chiral unitary with respect to $\gamma_0 \in \End_A(E)$, then so is $u^*$ and $-u$.
Furthermore, $u$ is of Fredholm type if and only if $u^*$ is of Fredholm type and~$\big[ \calC(u^*) \big]$ represents the inverse of $\big[ \calC(u) \big]$ in $KK(\C, A)$.
In contrast,
there is no relation between~$\big[\calC(u)\big]$ and~$\big[\calC(-u)\big]$ in
general. Indeed, there is no guarantee that $\calC(-u)$ is Fredholm if $\calC(u)$ is Fredholm.
This is analogous to the indices $\Ind(P_0, P_1)$ and $\Ind(P_0, \one - P_1)$ considered in Section~\ref{sec:HSpace_chiral_unitaries}.
\end{Remark}

To summarise the contents of this subsection, we have defined an index taking values in $KK(\C, A)\cong K_0(A)$
for chiral unitaries $u = \gamma_0 \gamma_1$ acting on the Hilbert $C^*$-module $E_A$ that are of Fredholm type
(see Definition \ref{Def:u_Fred_type}). Proposition \ref{prop:homotopy_of_chiral_unitaries} shows that this index is
indeed a~homotopy invariant of chiral unitaries on $E_A$ of Fredholm type.
Recalling Proposition \ref{Prop:Hspace_Cayley_P0P1}, the index~$\big[\calC(u)\big] \in KK(\C, A)$,
is a direct generalisation of the indices defined for chiral unitaries on Hilbert spaces.

\subsection{The Cayley transform of a pair of projections}

Because any chiral unitary $u \in \End_A(E)$ has a decomposition
$u = (2p_0- \one)(2p_1- \one)$ with $p_0, p_1 \in \End_A(E)$ projections, we can rewrite many of
our results concerning the Cayley transform in terms of $p_0$ and $p_1$ directly.
This Cayley transform can be seen as an equivalent formulation of the
 index of a pair of projections
considered in Section \ref{subsec:Index_pair_projs}.
We first note the following, which can be proved with basic algebra.

\begin{Lemma}[{cf.~\cite[Section 2]{ASS}}] \label{lem:p_q_properties}
Let $p_0, p_1 \in \End_A(E)$ be projections.
\begin{itemize}\itemsep=0pt
 \item[$(i)$] Then the operators $(p_1+p_0- \one)$ and $(p_1-p_0)$ anti-commute and $(p_1-p_0)^2 + (p_1+p_0- \one)^2 = \one$.
 \item[$(ii)$] The operator $(p_1-p_0)^2$ commutes with $p_0$ and $p_1$.
\end{itemize}
\end{Lemma}

\begin{Definition}
Let $p_0, p_1 \in \End_A(E)$ be projections. We define the Cayley transform of a~pair of projections as the operator
\[
 \calC(p_1,p_0) = {\rm i}(p_1+p_0- \one)(p_1-p_0)^{-1}, \qquad \Dom\bigl( \calC(p_1,p_0) \bigr) = (p_1-p_0)E_A.
\]
\end{Definition}

\begin{Lemma} \label{lem:C(p_0,p_1)_basics}
The operator $\calC(p_1,p_0)$ is self-adjoint and regular operator on $\ol{(p_1-p_0)E}_A$ and anti-commutes with
$2p_0- \one$.
\end{Lemma}
\begin{proof}
Defining $u=(2p_0- \one)(2p_1- \one)$ where $(2p_0- \one)u(2p_0- \one) = u^*$, this is just a
restatement of Lemma \ref{lem:chiral_cayley_selfadj} after noting that
\begin{align*}
 &(\one +u) = \bigl(2p_1- \one + 2p_0- \one)(2p_1- \one) = 2(p_1+p_0- \one)(2p_1- \one), \\
 &(\one -u)^{-1} = \bigl((2p_1- \one - (2p_0- \one))(2p_1- \one) \bigr)^{-1} = \frac{1}{2} (2p_1- \one)^{-1}(p_1-p_0)^{-1}. \tag*{\qed}
\end{align*}\renewcommand{\qed}{}
\end{proof}

\begin{Lemma}
Let $p_0, p_1 \in \End_A(E)$ be projections.
If $\|p_0-p_1\|_{\calQ_A(E)} < 1$, then $\calC(p_1,p_0)$ is Fredholm.
\end{Lemma}
\begin{proof}
We first note that ${\rm i} + \calC(p_1,p_0) = {\rm i}(2p_1- \one)(p_1-p_0)^{-1}$ on $\Dom(\calC(p_1,p_0))$ and so
$({\rm i}+ \calC(p_1,p_0))^{-1} = -{\rm i}(p_1-p_0)(2p_1- \one)$. Therefore, if $\|p_0-p_1\|_{\calQ_A(E)} < 1$, then
\[
 \big\| ({\rm i}+ \calC(p_1,p_0))^{-1} \big\|_{\calQ_{A}( \ol{(p_1-p_0)E} ) } = \| (-{\rm i})(p_1-p_0)(2p_1- \one) \|_{\calQ_A(E)}
 = \|p_0-p_1\|_{\calQ_A(E)} < 1.
\]
The result then follows by Lemma \ref{lem:unbdd_Fred_condition}.
\end{proof}

\begin{Remark}
Let us again briefly study the bounded transform
$F_{\calC(p_1,p_0)} = \calC(p_1,p_0)\bigl( \one + \calC(p_1,p_0)\bigr)^{-1/2}$. Using Lemma \ref{lem:p_q_properties}, we have that
\begin{align*}
 \one + \calC(p_1,p_0)^2 &= \one+ (p_1+p_0- \one)^2(p_1-p_0)^{-2}
 = \bigl( (p_1-p_0)^2 + (p_1+p_0- \one)^2 \bigr) (p_1-p_0)^{-2} \\
 &= (p_1-p_0)^{-2}
\end{align*}
and so $F_{\calC(p_1,p_0)} = {\rm i}(p_1+p_0- \one)(p_1-p_0)^{-1}|p_1-p_0| = {\rm i}(p_1+p_0- \one)V$ with $V$
the completion of the operator $V(p_1-p_0)e = |p_1-p_0|e$ on $(p_1-p_0)E_A$ and $e \in E_A$.
We also note that
\begin{align*}
 \one - F_{\calC(p_1,p_0)}^2 &= \one + (p_1+p_0- \one)(p_1-p_0)^{-1}|p_1-p_0| (p_1+p_0- \one)(p_1-p_0)^{-1}|p_1-p_0| \\
 &= \one - (p_1+p_0- \one)^2 (p_1-p_0)^{-2} |p_1-p_0|^2 \\
 &= (p_1-p_0)^2 + (p_1+p_0- \one)^2 - (p_1+p_0- \one)^2 = (p_1-p_0)^{2}.
\end{align*}
Therefore, $\pi\bigl(F_{\calC(p_1,p_0)}\bigr)$ is unitary in $\calQ_A\bigl( \ol{(p_1-p_0)E} \bigr)$ if and only
if $p_1-p_0 \in \mathbb{K}_A(E)$.
\end{Remark}

The self-adjoint unitary $2p_0- \one$ acts as a $\Z_2$-grading operator for the Hilbert
$C^*$-module $\ol{(p_1-p_0)E}_A$.
We therefore see that for a pair of projections $p_0$ and $p_1$ in $\End_A(E)$ with $\|p_1-p_0\|_{\calQ_A(E)} < 1$,
we can construct a Kasparov module and equivalence class $\big[\calC(p_1, p_0)\big] \in KK(\C,A)$.

Let us now relate the index $\big[\calC(p_1, p_0)\big]$ to the index for chiral unitaries.

\begin{Proposition}
Let $p_0, p_1 \in \End_A(E)$ be projections. Then $p_0$ and $p_1$ are a Fredholm pair of projections
if and only if $u = (2p_0- \one)(2p_1- \one) \in \End_A(E)$ is a chiral unitary with respect to~$(2p_0 - \one)$
of Fredholm type and
\begin{equation} \label{eq:C(u)_and_C(p,q)}
 \big[ \calC(p_1, p_0) \big] = \big[ \calC( u) \big] \in KK(\C, A), \qquad u = (2p_0- \one)(2p_1- \one).
\end{equation}
\end{Proposition}
\begin{proof}
We have already seen in Lemma \ref{lem:C(p_0,p_1)_basics} that
$\calC(p_1, p_0) = \calC(u)$ as unbounded operators on~$E_A$.
Noting that $p_1-p_0 = \tfrac{1}{2} (2p_0 - \one)(u- \one)$,
$\|p_1-p_0\|_{\calQ_A(E)} < 1$ if and only if $\|u - \one\|_{\calQ_A(E)} < 2$.
The result then follows.
\end{proof}

We also have that $\big[\calC( p_0, p_1) \big] = \big[ \calC(u^*) \big] = \big[ \calC(p_1,p_0) \big]^{-1}$ and
$\big[ \calC(-u) \big]=\big[ \calC( p_1, \one - p_0) \big]$ as elements of $KK(\C, A)$.
We can also apply the results of Proposition \ref{prop:homotopy_of_chiral_unitaries} to the pair of projections
setting. Namely, if $[0,1]\ni t \mapsto p_0(t)$ and $[0,1]\ni t \mapsto p_1(t)$ are strongly continuous paths of
projections in $\End_A(E)$ such that $\| p_0(t) - p_1(t) \|_{\calQ_A(E)}< 1$ for all $t\in [0,1]$, then~$\big[ \calC(p_1(t), p_0(t)) \big] \in KK(\C, A)$ will be constant.

Let us now consider the connection of $\big[ \calC(p_1, p_0) \big]$ to more standard
presentations of $K$-theory for operator algebras. Recall that
if $\|p_1-p_0\|_{\calQ_A(E)} < 1$, then there is a well-defined class in relative $K$-theory
$[p_1]-[p_0] \in K_0( \End_A(E), \mathbb{K}_A(E)) \cong K_0(\mathbb{K}_A(E))$.
We relate this relative $K$-theory class to $\big[ \calC(p_1, p_0) \big]$ by the following result.

\begin{Proposition} \label{prop:Cayley_proj_iso}
There is an isomorphism
$\mathfrak{C}\colon K_0( \End_A(E), \mathbb{K}_A(E)) \to KK(\C,A)$ such that
for $p_0, p_1 \in \End_A(E)$ projections with $\|p_1-p_0\|_{\calQ_A(E)} < 1$,
\[
 \mathfrak{C} \bigl( [p_1]-[p_0] \bigr) = \big[ \calC(p_1,p_0) \big].
\]
\end{Proposition}
\begin{proof}
To construct the isomorphism $\mathfrak{C}$, some care is required as the relative $K$-theory
 class $[p_1]-[p_0] \in K_0( \End_A(E), \mathbb{K}_A(E))$
considers $E_A$ as an
ungraded Hilbert $C^*$-module, whereas $\big[ \calC(p_1,p_0) \big] \in KK(\C, A)$ uses that
$E_A$ and $\ol{(p_1-p_0)E}_A$ can be equipped with a $\Z_2$-grading via $\gamma_0 = 2p_0-\one$.
We can amend this discrepancy by considering the Hilbert $C^*$-module
$E_A \hox \Cl_1$ with $\Cl_1$ the $\Z_2$-graded Hilbert $C^*$-module over itself. There is then a
$\Z_2$-graded isomorphism from $E_A \hox \Cl_1$ (with $E_A$ $\Z_2$-graded) to
$E_A \otimes \Cl_1$ (with $E_A$ ungraded) given by $e \hox \rho^k \mapsto e \otimes \rho^{k + |e|}$,
where $\rho$ is the odd self-adjoint generator of $\Cl_1$, $k \in \Z$ and $e \in E_A$ has homogeneous
grading $|e| \in \{0,1\}$. Also note that $\End_{A \otimes \Cl_1}( E \otimes \Cl_1) \cong \End_A(E) \otimes \Cl_1$,
where $\End_A(E)$ is an ungraded algebra.

We define $\mathfrak{C}$ by the following composition
\[
 K_0( \End_A(E), \mathbb{K}_A(E))
 \to KK(\Cl_1, A \otimes \Cl_1) \to KK(\C, A),
\]
where each map is an isomorphism.
For the first map, we use the Cayley isomorphism of stable homotopy classes of odd self-adjoint
unitaries considered in~\cite[Section 4]{BKR2}. Namely, given
$p_0, p_1 \in \End_A(E)$ projections with $\|p_1-p_0\|_{\calQ_A(E)} < 1$,
we can consider the self-adjoint odd unitaries
$(2p_1- \one)\otimes \rho, (2p_0- \one)\otimes \rho \in \End_A(E)\otimes \Cl_1$.
Applying the results of~\cite[Section 4]{BKR2}, we have an isomorphism
$K_0( \End_A(E), \mathbb{K}_A(E)) \to KK(\Cl_1, A \otimes \Cl_1)$ given by
\[
 [p_1]-[p_0] \mapsto
 \big[\bigl( \Cl_1, \ol{(p_1-p_0)E}_A \otimes \Cl_1, (2p_0- \one)( p_1+p_0-\one) (p_1-p_0)^{-1} \otimes \rho \bigr) \big],
\]
where the left $\Cl_1$-action on $\ol{(p_1-p_0)E}_A \otimes \Cl_1$ is generated by $\gamma_0 \otimes \rho$
and $(2p_0- \one)( p_1+p_0-\one) (p_1-p_0)^{-1} \otimes \rho$
is the relative Cayley transform of $(2p_1- \one)\otimes \rho$ and $(2p_0- \one)\otimes \rho$.

To recover $\big[ \calC(p_1, p_0) \big]$, we now restore the $\Z_2$-grading on $\ol{(p_1-p_0)E}_A$,
which can be done by the $\Z_2$-graded isomorphism $\ol{(p_1-p_0)E}_A \otimes \Cl_1 \to \ol{(p_1-p_0)E}_A \hox \Cl_1$,
where $e \otimes \rho^k \mapsto e \hox \rho^{k+|e|}$,
$k \in\Z$ and $e \in \ol{(p_1-p_0)E}_A$ is such that $\gamma_0 e = (-1)^{|e|}e$. Note also that
this isomorphism is compatible with the domain of $(2p_0- \one)( p_1+p_0- \one) (p_1-p_0)^{-1}\otimes \rho$.
The isomorphism $\ol{(p_1-p_0)E}_A \otimes \Cl_1 \cong \ol{(p_1-p_0)E}_A \hox \Cl_1$ gives us the
unitarily equivalent Kasparov module
\[
\big[ \bigl( \Cl_1, \ol{(p_1-p_0)E}_A \hox \Cl_1 , (2p_0- \one)( p_1+p_0 - \one) (p_1-p_0)^{-1} \hox \one \bigr) \big],
\]
which we can now split into an external product
\[
 \big[ \bigl( \C, \ol{(p_1-p_0)E}_A , (2p_0- \one)( p_1+p_0 - \one) (p_1-p_0)^{-1} \bigr) \big] \hat\otimes_\C
 \big[ ( \Cl_1, \Cl_1, 0 ) \big].
\]
The equivalence class
$\big[ ( \Cl_1, \Cl_1, 0 ) \big]$ is the ring identity in $KK(\Cl_1,\Cl_1)$ and therefore this splitting implements
the isomorphism $KK\bigl(\Cl_1, A \hox \Cl_1\bigr)\xrightarrow{\simeq} KK(\C, A)$.

Composing the above operations gives the map $\mathfrak{C}\colon K_0( \End_A(E), \mathbb{K}_A(E)) \to KK(\C,A)$, i.e.,
\[
 \mathfrak{C} \bigl( [p_1]-[p_0] \bigr) = \big[ \bigl( \C, \ol{(p_1-p_0)E}_A , (2p_0- \one)( p_1+p_0 - \one) (p_1-p_0)^{-1} \bigr) \big].
\]
Finally, we observe that the operators $(2p_0- \one)( p_1+p_0 - \one) (p_1-p_0)^{-1}$ and $i(p_1+p_0- \one)(p_1-p_0)^{-1}$ have the
same domain of definition and anti-commute. Hence they will define the same class in~$KK(\C, A)$ by
Proposition \ref{prop:anticommuting_Fredholms_homotopic}. Thus $\mathfrak{C} \bigl( [p_1]-[p_0] \bigr) = \big[ \calC(p_1,p_0) \big]$.
\end{proof}

Recalling Section \ref{subsec:Index_pair_projs}, if $E_A$ is a full Hilbert $C^*$-module and
$\|p_1-p_0\|_{\calQ_A(E)} < 1$ then
 there is also an isomorphism $K_0( \End_A(E), \mathbb{K}_A(E)) \xrightarrow{\simeq} KK(\C,A)$ that maps
 $[p_1] - [p_0]$ to
\[
 [(p_1, p_0)] = \bigg[ \biggl( \C, p_1 E_A \oplus p_0 E_A, \begin{pmatrix} 0 & p_1 V^* p_0 \\ p_0 V p_1 & 0 \end{pmatrix} \biggr) \bigg]
 \in KK(\C, A),
\]
with $V \in \End_A(E)$ such that $\pi(V) \in \calQ_A(E)$ is unitary and
$\pi( V p_1 V^*) = \pi(p_0)$.
A direct comparison of $\big[\calC(p_1,p_0)\big]$ and $[(p_1, p_0)]$ in $KK(\C, A)$ is difficult, but we can at
least say the following.

\begin{Corollary} \label{Corollary:Ind(p,q)_to_C(p,q)}
If $E_{\!A}$ is a full Hilbert $C^*$-module, then there is an automorphism $\phi$ of $K\!K(\C,\! A)$ such
that $\phi\bigl( [(p_1, p_0)] \bigr) = \big[ \calC(p_1, p_0) \big]$.
\end{Corollary}
\begin{proof}
We define $\phi$ as $\mathfrak{C} \circ \zeta$ with $\mathfrak{C}$ the isomorphism from
Proposition \ref{prop:Cayley_proj_iso} and $\zeta$ the composition of the Morita equivalence
$KK(\C, A) \cong KK(\C, \mathbb{K}_A(E))$ with the inverse of the isomorphism from Theorem \ref{Theorem:KK_to_K_Kasparov}.
\end{proof}

Recalling that $\big[ \calC(p_1, p_0) \big] = \big[ \calC( u) \big]$ for $u = (2p_0- \one)(2p_1- \one)$,
our results show that the $K$-theory group $K_0(A)$ can be regarded as stable homotopy classes of
chiral unitaries acting on Hilbert $C^*$-modules over $A$. The Cayley transform provides an isomorphism
of such unitaries to $KK(\C, A)$.
As our results show, the Cayley transform on Hilbert $C^*$-modules provides
a~convenient framework to pass between various pictures of $K$-theory for operator algebras and
so may
be useful for future studies in noncommutative index theory.
Similar results for graded and real $K$-theory were also considered in~\cite{B21, BKR2}.

\subsection{Connection to the total symmetry index of a chiral unitary}

To better connect our indices on Hilbert $C^*$-modules to the Hilbert space indices studied in
Section \ref{sec:HSpace_chiral_unitaries},
we now define an analogue of the index $\operatorname{si}(U) = \Ind_{\Gamma_0}(U)$ in the Hilbert $C^*$-module setting.
Like the Hilbert space setting, we will show that this index can be decomposed into a sum of the
indices that we have defined using the Cayley transform.
Like the previous subsections, we will
work in a fixed countably
generated Hilbert $C^*$-module $E_A$ over a $\sigma$-unital $C^*$-algebra~$A$.

\begin{Proposition}
Let $u \in \End_A(E)$ be a chiral unitary with respect to $\gamma_0 \in \End_A(E)$. If
$\| u+ u^*\|_{\calQ_A(E)} < 2$, then $\tfrac{1}{2\rm{i}}(u-u^*)$ is a self-adjoint Fredholm operator in $\End_A(E)$
that anti-commutes with $\gamma_0$.
\end{Proposition}
\begin{proof}
Letting $T = \tfrac{1}{2\rm{i}}(u-u^*)$, we see that
\[
 \one - T^2 = \one + \frac{1}{4}\bigl(u^2 + u^{-2} -2\bigr) = \frac{1}{4}\bigl(2 + u^2 + u^{-2}\bigr) = \frac{1}{4}(u+u^*)^2.
\]
In particular $\big\| \one - T^2 \|_{\calQ_A(E)} = \tfrac{1}{4} \bigl( \|u+u^*\|_{\calQ_A(E)}\bigr)^2$ and $\pi(T)\in \calQ_A(E)$
is invertible if $\| u+ u^*\|_{\calQ_A(E)} < 2$ by Lemma \ref{lem:unbdd_Fred_condition}. The property that
$T=T^*$ is immediate and $T\gamma_0 + \gamma_0 T=0$ is a~simple check.
\end{proof}

Because $\tfrac{1}{2\rm{i}}(u-u^*)$ is an odd self-adjoint Fredholm operator on $E_A = (\one + \gamma_0)E_A \oplus (\one -\gamma_0)E_A$,
it defines a class $\big[ \tfrac{1}{2\rm{i}}(u-u^*) \big] \in KK(\C, A) \cong K_0(A)$.

Much like the Hilbert space setting in Section \ref{subsec:Ind(p,q)_HSpace}, we can relate $\big[ \tfrac{1}{2\rm{i}}(u-u^*) \big]$
to the sum of an index of a pair of projections in $E_A$.

\begin{Proposition} \label{prop:C*Mod_Suzuki_is_Ind(p,q)}
Let $u = (2p_0-\one)(2p_1- \one) \in \End_A(E)$ be a chiral unitary in $\End_A(E)$.
If~$\|p_0-p_1\|_{\calQ_A(E)} < \tfrac{1}{\sqrt{2}}$ and $\|p_0+p_1- \one\|_{\calQ_A(E)} < \tfrac{1}{\sqrt{2}}$, then
\[
 [(p_0, p_1)] \oplus [(p_0, \one -p_1)] = \big[ \tfrac{1}{2\rm{i}}(u-u^*) \big] \in KK(\C, A ).
\]
\end{Proposition}
\begin{proof}
We first note that
\[
 u + u^* = \frac{1}{2} \bigl( (u+ \one)^*(u+ \one) - (u- \one)^*(u- \one) \bigr) = 2\bigl( (p_0+p_1- \one)^2 - (p_1-p_0)^2 \bigr),
\]
which we can use to estimate $\| u+ u^*\|_{\calQ_A(E)}\! < \!2$ if
$\|p_0-p_1\|_{\calQ_A(E)}\! <\! \tfrac{1}{\sqrt{2}}$ and ${\|p_0+p_1\!-\! \one\|_{\calQ_A(E)}\! <\!\! \tfrac{1}{\sqrt{2}}}$.
We also have unitaries $v, w \in \calQ_{A}(E)$
such that
\[
 v \pi(p_0) v^* = \pi(p_1), \qquad w \pi(p_1) w^* = \pi(\one - p_0).
\]
Under the decomposition $E_A = p_0 E_A \oplus (\one - p_0)E_A$, one finds that
\[
 \frac{1}{2\rm{i}}(u-u^*) = 2{\rm i}\begin{pmatrix} 0 & -p_0 p_1 (\one - p_0) \\ (\one - p_0) p_1 p_0 & 0 \end{pmatrix}
\]
(see the beginning of proof of Proposition \ref{prop:HSpace_SuzukiIndex_is_ProjectionIndex} for further details).
In particular, the class $\big[ \tfrac{1}{2\rm{i}}(u-u^*) \big] \in KK(\C,A)$ can be concretely expressed via the Kasparov module
\[
 \biggl( \C, p_0 E_A \oplus (\one - p_0) E_A, \begin{pmatrix} 0 & V^* p_1 W^* \\ W p_1 V & 0 \end{pmatrix} \biggr),
\]
where $V, W \in \End_A(E)$ are lifts of $v$ and $w \in \calQ_A(E)$, respectively.
However, recalling Proposition~\ref{prop:Index_pair_projection_sum_Kasmod}, this Kasparov module precisely represents
the sum $[(p_0, p_1)] \oplus [(p_1, \one - p_0)] \in KK(\C, A)$. Applying part (ii) of Lemma \ref{lem:[(p,q)]_properties},
where therefore have that
\[
 \big[ \tfrac{1}{2\rm{i}}(u-u^*) \big] = [(p_0, p_1)] \oplus [(p_1, \one - p_0)] = [(p_0, p_1)] \oplus [(p_0, \one -p_1)]. \tag*{\qed}
\]\renewcommand{\qed}{}
\end{proof}

\begin{Remarks}\quad
\begin{itemize}\itemsep=0pt
 \item[(i)] Let ${\rm Inv}$ be the automorphism $[x]\mapsto [x]^{-1}$ on $KK(\C, A)$. Then supposing
 that $E_A$ is a full Hilbert $C^*$-module, we can use the automorphism $\phi$ of $KK(\C, A)$ from
 Corollary \ref{Corollary:Ind(p,q)_to_C(p,q)} and equation \eqref{eq:C(u)_and_C(p,q)} on p.~\pageref{eq:C(u)_and_C(p,q)}
 to infer that
 \begin{align*}
 ( {\rm Inv}\circ \phi) \bigl( \big[ \tfrac{1}{2\rm{i}}(u-u^*) \big] \bigr)& = \big[ \calC(p_1, p_0) \big] \oplus \big[ \calC(\one - p_1, p_0) \big]\\
 & = \big[ \calC(u) \big] \oplus \big[ \calC( -u) \big ] \in KK(\C, A).
 \end{align*}
 \item[(ii)] The hypotheses $\|p_0-p_1\|_{\calQ_A(E)} < \tfrac{1}{\sqrt{2}}$ and $\|p_0+p_1- \one\|_{\calQ_A(E)} < \tfrac{1}{\sqrt{2}}$
 in Proposition \ref{prop:C*Mod_Suzuki_is_Ind(p,q)}
 are most likely far from optimal and we expect that assumed bounds
 can be weakened to $\|p_0-p_1\|_{\calQ_A(E)} < 1$ and
 $\|p_0+p_1- \one\|_{\calQ_A(E)} < 1$. We leave this question to another place.
\end{itemize}
\end{Remarks}

\subsection{Connection to the generator/Hamiltonian} \label{subsec:Module_generator_index}

We once again consider the countably generated Hilbert $C^*$-module $E_A$, a
self-adjoint regular operator $H\colon \Dom(H) \to E_A$ and a self-adjoint
unitary $\gamma_0 \in \End_A(E)$ such that
\[
 \gamma_0 \cdot \Dom(H) \subset \Dom(H),
 \qquad \gamma_0 H \gamma_0 = -H.
\]

We can then define the unitaries $\pm {\rm e}^{{\rm i}\pi H} \in \End_A(E)$, which play the role of
 discrete time step operators and are chiral unitaries with respect to $\gamma_0$.

From the perspective of topological phases of matter, we are often interested in low-energy effects
and properties of Hamiltonians. Hence we may also wish to consider $\eta(H) \in \End_A(E)$ with $\eta\colon \R \to [-1,1]$ a
continuous, odd and non-decreasing function such that $\eta^{-1}(\{0\}) = \{0\}$ and
$\lim_{x\to \infty} \eta(x) = 1$.\footnote{The function $\eta$ is clearly very similar to a normalising function
considered in Definition \ref{def:Reg_fn}.}
Because $\eta$ is odd, $\gamma_0 \eta(H) = - \eta(H) \gamma_0$ and so ${\rm e}^{{\rm i} \pi \eta(H)}$ is a
chiral unitary with respect to $\gamma_0$.

We are particularly interested in the unitary $U_H := -{\rm e}^{{\rm i}\pi \eta(H)}$. Our sign convention is largely
motivated by the following result.

\begin{Theorem} \label{Theorem:Generator_and_unitary_same_class}
Let $H$ be a self-adjoint regular operator on $E_A$ anti-commuting with
a self-adjoint unitary $\gamma_0 \in \End_A(E)$.
If $H$ is Fredholm, then the chiral unitary
$U_H = -{\rm e}^{{\rm i} \pi \eta(H)}$ is of Fredholm type and
$[ H ] = \big[ \calC(U_H) \big] \in KK(\C, A) \cong K_0(A)$.
\end{Theorem}

We will prove this result in a few steps.

\begin{Lemma}
Define the operator $\wt{\calC}(H ) = (H-{\rm i})(H+{\rm i})^{-1}$. Then
$\wt{\calC}(H )$ is a unitary operator in $\End_A(E)$ such that
$\gamma_0 \wt{\calC}(H) \gamma_0 = \wt{\calC}(H)^*$. Furthermore, if $(H+{\rm i})^{-1} \in \mathbb{K}_A(E)$,
then $\wt{\calC}(H ) \in \mathbb{K}_A(E)^{\sim}$, the minimal unitisation of $\mathbb{K}_A(E)$.
\end{Lemma}
\begin{proof}
Because $H$ is self-adjoint and regular, its Cayley transform $\wt{\calC}(H)\in \End_A(E)$ is unitary~\cite[Theorem 10.5]{Lance}.
Using that $\gamma_0 \cdot \Dom(H) \subset \Dom(H)$, we compute
\[
 \gamma_0 \wt{\calC}(H ) \gamma_0 = \gamma_0 (H-{\rm i}) \gamma_0 \bigl( \gamma_0(H+{\rm i} ) \gamma_0 \bigr)^{-1}
 = (H+{\rm i})(H-{\rm i} )^{-1} = \wt{\calC}(H)^*.
\]
If $(H+{\rm i})^{-1} \in \mathbb{K}_A(E)$, then
\[
 \one - \wt{\calC}(H ) = \bigl( H+{\rm i} - (H-{\rm i} ) \bigr)(H+{\rm i} )^{-1} = 2{\rm i} (H+{\rm i} )^{-1} \in \mathbb{K}_A(E). \tag*{\qed}
\]\renewcommand{\qed}{}
\end{proof}

\begin{Remark}
If $H$ is bounded, $-1$ is not in the spectrum of $\wt{\calC}(H )$. If $H$ is invertible (with bounded inverse), $1$
is not in the spectrum of $\wt{\calC}(H )$.
\end{Remark}

\begin{Lemma} \label{lem:TFred_iff_CTFred}
A self-adjoint regular operator $H$ is Fredholm on $E_A$ if and only if $\wt{\calC}(H)$ is of Fredholm type.
\end{Lemma}
\begin{proof}
Recalling Lemma \ref{lem:unbdd_Fred_condition}, $H$ is Fredholm on $E_A$ if and only if
$\big\| (H+{\rm i})^{-1} \big\|_{\calQ_A(E)} < 1$. We now use that $\one - \wt{\calC}(H) = 2{\rm i}(H+{\rm i})^{-1}$, so
$\big\| \one - \wt{\calC}(H) \big\|_{\calQ_A(E)} < 2$ if and only if $\big\| (H+{\rm i})^{-1} \big\|_{\calQ_A(E)} < 1$.
\end{proof}

\begin{Lemma} \label{lem:Fredholm_then_UT_Fredholm}
If $H$ is Fredholm, then $U_H =-{\rm e}^{{\rm i} \pi \eta(H)}$ is a chiral unitary of Fredholm type.
\end{Lemma}
\begin{proof}
By the definition of $\eta(H)$, we have that $\one - U_H = \one + {\rm e}^{{\rm i} \pi \eta(H)} = g(H)$ with $g \in C_0(\R)$ such that
$\|g\|_\infty = 2$ and $g^{-1}(\{ 2,-2\}) = g^{-1}(\{ 2\}) = \{0\}$.
Because $H$ is Fredholm,
by Proposition~\ref{prop:Fredholm_properties} there exists some $\varepsilon >0$ such that
$h(H) \in \mathbb{K}_A(E)$ for all $h \in C_c( -\varepsilon, \varepsilon)$. We can therefore decompose
$g(H) = \tilde{g}(H) + g_\varepsilon(H)$ with $g_\varepsilon \in C_c( -\varepsilon, \varepsilon)$ and
$\| \tilde{g} \|_\infty < 2$. Thus we have the inequality
\[
 \| \one - U_H\|_{\calQ_A(E)} = \| g(H) \|_{\calQ_A(E)} = \| \tilde{g}(H) \|_{\calQ_A(E)} < 2
\]
as required.
\end{proof}

If the self-adjoint operator $H$ has compact resolvent, $(H+{\rm i})^{-1} \in \mathbb{K}_A(E)$,
then $g(H) \in \mathbb{K}_A(E)$ for any $g \in C_0(\R)$. In particular, $U_H =-{\rm e}^{{\rm i} \pi \eta(H)}$ will be in the minimal
unitisation of $\mathbb{K}_A(E)$ with $\| \one - U_H \|_{\calQ_A(E)} = 0$.

\begin{Lemma}
The unitaries $U_H =-{\rm e}^{{\rm i} \pi \eta(H)}$ and $\wt{\calC}(H ) = (H-{\rm i})(H+{\rm i})^{-1}$ are homotopic in $\End_A(E)$.
If $H$ is Fredholm on $E_A$, then they are homotopic via a path of chiral unitaries of Fredholm type.
\end{Lemma}
\begin{proof}
It is simple to check that both unitaries asympotically behave like $1$ at $\pm \infty$ and will wind once
around the one-point compactification. Hence
they are homotopic via a strongly continuous path of unitaries $\{u_t\}_{t \in [0,1]} \subset \End_A(E)$.
We can write this homotopy as $u_t = -{\rm e}^{{\rm i}\pi \eta_t(H)}$, where $t\mapsto \eta_t$ is a
path of continuous odd non-decreasing functions such that for all $t\in [0,1]$, $\eta_t^{-1}(\{0\}) = \{0\}$
and $\lim_{x\to \infty} \eta_t(x) = 1$.
Therefore, if $H$ is Fredholm on $E_A$, then
by Lemma \ref{lem:Fredholm_then_UT_Fredholm}, $u_t$ will be a chiral unitary of Fredholm type for
every $t\in [0,1]$.
\end{proof}

\begin{proof}[Proof of Theorem \ref{Theorem:Generator_and_unitary_same_class}]
By Lemma \ref{lem:TFred_iff_CTFred}, $H$ is Fredholm on $E_A$ if and only if
$\wt{\calC}(H)$ is a chiral unitary of Fredholm type, which in turn implies that
$\calC\bigl( \wt{\calC}(H) \bigr)$ is a self-adjoint Fredholm operator on
$\ol{\bigl(\wt{\calC}(H)-1\bigr)E}_A$
anti-commuting with $\gamma_0$.
We first note that the domain of $\calC\bigl( \wt{\calC}(H) \bigr)$ is
\[
 \bigl( \one - \wt{\calC}(H)\bigr)E_A = 2{\rm i}(H+{\rm i})^{-1}E_A = \Dom(H)
\]
and on this domain
\begin{align*}
 &\calC\bigl( \wt{\calC}(H) \bigr) = {\rm i}\bigl( \one + \wt{\calC}(H)\bigr) \bigl( \one - \wt{\calC}(H) \bigr)^{-1}= {\rm i}(H +{\rm i} + (H-{\rm i})) ( H+{\rm i} - (H-{\rm i}))^{-1} = H.
\end{align*}
Hence $[ H ] = \big[ \calC\bigl( \wt{\calC}(H) \bigr) \big]\! \in\! KK(\C, A)$.
Finally, by Lemma \ref{lem:TFred_iff_CTFred}, $\wt{\calC}(H)$ is homotopic to $-{\rm e}^{{\rm i} \pi \eta(H)}$ via
a path of chiral unitaries of Fredholm type.
Therefore, by Proposition \ref{prop:homotopy_of_chiral_unitaries},
$[ H] = \big[ \calC\bigl( \wt{\calC}(H) \bigr) \big] = \big[ \calC( U_H) \big]$.
\end{proof}

\begin{Remark}
Indices of the form $[H] \in KK(\C, A)$ often appear in the study of higher index theory or systems with
a boundary. Theorem \ref{Theorem:Generator_and_unitary_same_class} shows that we can recast this
index problem in terms of homotopy equivalence classes of unitary operators. Depending on the situation
under study, it may be more tractable to work with unitary operators rather than (possibly unbounded) self-adjoint
operators. Hence our result may offer additional insight into the class~$[H]$. We leave a more thorough
study of this question to future work.
\end{Remark}

\subsection{Some brief remarks on non-chiral unitaries and quantum walks}

Split-step quantum walks represent a special subclass of more general quantum walks, unitary operators
$U$ on $\calH$ that have a decomposition into shift and coin operators.\footnote{See~\cite{Cedzich22b}
for an algorithm to decompose any banded unitary acting on a one-dimensional lattice structure into shift and
coin operators.}
However, index theoretic properties of generic quantum walk unitaries are seldom studied as
any two unitaries in $\calB(\calH)$ or the minimal unitisation of $\calK(\calH)$ are stably homotopic.
By expanding our domain of definition of quantum walks to include unitaries on Hilbert $C^*$-modules $E_A$,
we are able to access non-trivial index-theoretic invariants via $KK^1(\C, A) \cong K_1(A)$,
the odd $K$-theory of operator algebras.

Let us fix a $\sigma$-unital $C^*$-algebra $A$ and a countably generated Hilbert $C^*$-module $E_A$.
Our results in Sections \ref{subsec:Module_Cayley} and \ref{subsec:Module_generator_index} transfer to the
setting of non-chiral unitaries in $\End_A(E)$ up to a degree shift in $K$-theory. That is, we are
considering classes in $KK^1(\C, A) \cong K_1(A)$, though the proofs are the same but without the grading
operator $\gamma_0$. We provide a brief summary.

\begin{Definition}
We say that a unitary $u \in \End_A(E)$ is of Fredholm type if $\| u - \one \|_{\calQ_A(E)} < 2$.
\end{Definition}

\begin{Proposition}
Let $u \in \End_A(E)$ be a unitary of Fredholm type. Then the operator
\[
 \calC(u ) = {\rm i}(\one +u)(\one -u)^{-1} , \qquad \Dom(\calC(u)) = (\one -u)E_A
\]
is self-adjoint, regular and Fredholm on $\ol{(u-\one)E}_A$ and so defines a class
$\big[ \calC(u) \big] \in KK^1(\C, A) \cong K_1(A)$. If $\{u_t\}_{t\in [0,1]} \subset \End_A(E)$ is a
strongly continuous path of unitaries of Fredholm type, then $\big[ \calC(u_t) \big] \in KK^1(\C, A)$ is constant.
\end{Proposition}

\begin{Theorem}
Let $\eta\colon \R \to \R$ be a continuous, odd and non-decreasing function such that
$\eta^{-1}(\{0\}) = \{0\}$ and
$\lim_{x\to \infty}\eta(x) = 1$.
If $H$ is a self-adjoint regular Fredholm operator on $E_A$, then
$U_H = -{\rm e}^{{\rm i} \pi \eta(H)} \in \End_A(E)$ is a unitary of Fredholm type and
$[ H ] = \big[ \calC(U_H) \big] \in KK^1( \C, A)$.
\end{Theorem}

Let us consider a simpler setting, where the unitary $u \in \mathbb{K}_A(E)^\sim$, the minimal
unitisation of $u \in \mathbb{K}_A(E)^\sim$, and so $\| u - \one \|_{\calQ_A(E)} = 0$.
Because $\mathbb{K}_A(E)$ is Morita equivalent to an
ideal of~$A$, we can directly consider the class $[u] \in K_1(A)$ without applying the Cayley
transform.
The following result reconciles these two approaches.

\begin{Theorem}[{\cite[Theorem 3.5]{BKR2}}]
The map
\[
 K_1(A) \ni [u] \mapsto \big[ \calC(u) \big] \in KK^1(\C, A)
\]
is well defined and an isomorphism of groups.
\end{Theorem}

\begin{Example}[quantum walks with weighted shifts]
Let us consider the following simple quantum walk unitary on $\ell^2\bigl(\Z, \C^2\bigr)$,
\[
 U_{m,n} = S_{m,n} C, \qquad S_{m,n} = \begin{pmatrix} S^m & 0 \\ 0 & (S^*)^n \end{pmatrix},
 \qquad C \in \calU\bigl(\C^2\bigr), \qquad m, n \in \Z,
\]
where $S$ is the shift operator on $\ell^2(\Z)$, $Se_k = e_{k+1}$ with $\{e_k\}_{k\in \Z}$ the
standard orthonormal basis of $\ell^2(\Z)$. The operator $U_{m,n}$ can also be considered as
a unitary that acts on the Hilbert $C^*$-module $E_{C^*(\Z)} = C^*(\Z)_{C^*(\Z)} \otimes \C^2$
by left-multiplication. Because $C^*(\Z)$ is unital,
we have that $\End_{C^*(\Z)}\bigl( C^*(\Z) \otimes \C^2\bigr) = \mathbb{K}_{C^*(\Z)}\bigl(C^*(\Z) \otimes \C^2\bigr) = M_2(C^*(\Z))$.
So in particular the condition $\| U_{m,n} - \one \|_{\calQ} < 2$ is trivially satisfied and
$U_{m,n}$ defines a class $[U_{m,n}] \in K_1(C^*(\Z))$. We can take a path of unitary matrices connecting
$C$ to $\one$ in $\calU(\C^2)$, therefore
\[
 [ S_{m,n} C ] = [ S_{m,n} ] = [S^m] \oplus [(S^*)^{n} ] \in K_1(C^*(\Z)).
\]
Hence, the isomorphism $K_1(C^*(\Z)) \xrightarrow{\simeq} \Z$ is such that $[U_{m,n}] \mapsto m-n$.
\end{Example}

\section{An index formula for anisotropic split-step quantum walks} \label{sec:1D_winding}

In this section, we provide an extension of the winding number formula for split-step quantum walks by
Matsuzawa~\cite{Matsuzawa} to the setting of semifinite index theory~\cite{CPRS2}.
We have previously defined a $K$-theoretic index $[\calC(u)] \in KK(\C, A)$ for a chiral
unitary of Fredholm type on a~Hilbert $C^*$-module $E_A$. Here we consider the case of an
index $[\calC(u)] \in KK(\C, B)$ with $B$ a~unital $C^*$-algebra with a continuous trace
$\tau_B\colon B \to \C$. This allows us to consider the $\R$-valued index
$(\tau_B)_\ast \bigl( [\calC(u)] \bigr)$ using the map $(\tau_B)_\ast\colon KK(\C, B) \to \R$ from
Section \ref{subsec:induced_trace}. As we will show, if we consider the Hilbert $C^*$-module
$\ell^2(\Z, B)$ and assume an anisotopy condition on the chiral unitary $u$, then we can compute
$(\tau_B)_\ast \bigl( [\calC(u)] \bigr)$ via the noncommutative winding number of a pair of unitaries
at the limits $\pm \infty$.

\subsection{Preliminaries and setting}

Let $B$ be a {unital} $C^*$-algebra with an automorphism $\alpha \in \Aut(B)$.
We can define the crossed product $C^*$-algebra $B \rtimes_\alpha \Z$, which is
the $C^*$-closure of $B$ and a unitary element $S$ such that $Sb = \alpha(b) S$.

We now consider the Hilbert $C^*$-module $\ell^2(\Z, B) \cong \ell^2(\Z) \otimes B$, where we
wish to consider quantum walk-like operators. The usual setting of quantum walks can be recovered
by taking $B = \C^n$ or $M_n(\C)$.

Let $C(\Z \cup \{\pm \infty\}, B)$ be the $C^*$-algebra of functions $f\colon \Z \to B$ such that the limits
at $\pm \infty$ exist.
 The algebra $C(\Z \cup \{\pm \infty\}, B)$ also comes with an automorphism $\wt{\alpha}$,
 \[
 \wt{\alpha}( f )(x) = \alpha[ f(x-1) ], \quad x\in \Z, \qquad \wt{\alpha}( f )(\pm \infty) = \alpha[ f(\pm \infty) ],
 \quad f \in C(\Z \cup \{\pm \infty\}, B).
 \]
We can therefore consider the
crossed product $C(\Z \cup \{\pm \infty\}, B) \rtimes_{\wt{\alpha}} \Z$, which is generated by
$C(\Z \cup \{\pm \infty\}, B)$ and a unitary $\wt{S}$ that implements $\wt{\alpha}$.

There is an action of $C(\Z \cup \{\pm \infty\}, B)$ on $B$ given by $f \cdot b = f(0) b$ for
$f \in C(\Z \cup \{\pm \infty\}, B)$ and $b \in B$. This action can be extended to give a representation
of $C(\Z \cup \{\pm \infty\}, B) \rtimes_{\wt{\alpha}} \Z$ on~$\ell^2(\Z) \otimes B$, where for $f \in C(\Z \cup \{\pm \infty\}, B)$,
$b \in B$ and $\{e_j\}_{j\in \Z}$ the standard orthonormal basis of $\ell^2(\Z)$,
\begin{equation} \label{eq:crossed_prod_action}
 \wt{S}^n f\cdot (e_j \otimes b) = e_{j+n} \otimes \wt{\alpha}_{-j}(f)\cdot b = e_{j+n} \otimes \wt{\alpha}_{-j}(f)(0)b
 = e_{j+n} \otimes \alpha_{-j}[f(j)] b.
\end{equation}

\begin{Lemma}[{\cite[Proposition 3.1]{BKR1}}] \label{Lemma:cross_prod_acts_on_module}
Equation \eqref{eq:crossed_prod_action} extends to a
$\ast$-homomorphism
\[C(\Z \cup \{\pm \infty\}, B) \rtimes_{\wt{\alpha}} \Z \to \End_B\bigl( \ell^2(\Z, B) \bigr).\]
\end{Lemma}

Lemma \ref{Lemma:cross_prod_acts_on_module}
implies that we can naturally consider $C(\Z \cup \{\pm \infty\}, B) \rtimes_{\wt{\alpha}} \Z$
as a subalgebra of~$\End_B(\ell^2(\Z, B)) \cong \Mult(\calK \otimes B)$ with $\calK$ the compact
operators on $\ell^2(\Z)$.

Let us briefly comment on the algebra $C(\Z \cup \{\pm \infty\}, B) \rtimes_{\wt{\alpha}} \Z$ and its relevance
for considering quantum walk-like operators.
Taking a chiral unitary $u = \gamma_0 \gamma_1 \in \End_B\bigl(\ell^2(\Z, B)\bigr)$,
we expect the unitaries $\gamma_0$ and $\gamma_1$ to be determined
by the shift operator $\wt{S}$ and functions $\Z \to B$. By assuming an anisotropic condition on the $B$-valued functions, they
can be considered as elements in $C(\Z \cup \{\pm \infty\}, B)$. If $\gamma_0$ and $\gamma_1$ contain
at most a finite polynomial of shift operators, then~$\gamma_0$,~$\gamma_1$ and therefore $u$ are elements of $C( \Z \cup \{\pm \infty\}, B) \rtimes_{\wt{\alpha}} \Z$.

Letting $C_0(\Z, B)$ be the functions $f\colon \Z \to B$ that vanish at $\pm \infty$, we can also consider
$C_0(\Z, B) \rtimes_{\wt{\alpha}} \Z$, which is a $2$-sided ideal in $C(\Z \cup \{\pm \infty\}, B) \rtimes_{\wt{\alpha}} \Z$.

\begin{Lemma}[{\cite[p.~147]{Rieffel82}}]
There is an isomorphism $C_0( \Z, B) \rtimes_{\wt{\alpha}} \Z \cong B \otimes \calK$.
\end{Lemma}
\begin{proof}[Proof sketch]
Writing $C_0( \Z, B) \cong C_0(\Z) \otimes B$, the action $\wt{\alpha}$ decomposes as $T\otimes \alpha$,
where $T$ is the translation action on $C_0(\Z)$.
We can then construct an isomorphism
\[
 \bigl( C_0(\Z) \otimes B \bigr) \rtimes_{T \otimes \alpha} \Z \xrightarrow{\simeq} \bigl( C_0(\Z) \otimes B \bigr) \rtimes_{T \otimes {\rm id}} \Z,
\]
where on the dense $\ast$-subalgebra $C_c( \Z, C_0(\Z, B) )$ this map is given by
\[
 g(x; y) \mapsto \alpha_{-y} [ g(x;y)], \qquad x, y \in \Z, \quad g(x;y) \in B.
\]
One then checks this map is compatible with the convolution product.
We therefore have the following chain of isomorphisms
\begin{align*}
 C_0( \Z, B) \rtimes_{\wt{\alpha}} \Z &\cong \bigl( C_0(\Z) \otimes B \bigr) \rtimes_{T \otimes \alpha} \Z
 \cong \bigl( C_0(\Z) \otimes B \bigr) \rtimes_{T \otimes {\rm id}} \Z \\
 &\cong \bigl(C_0(\Z) \rtimes_T \Z \bigr) \otimes B \cong \calK\bigl(\ell^2(\Z)\bigr) \otimes B,
\end{align*}
with the last isomorphism the Takai duality of $\Z$.
\end{proof}

We have that the quotient map
$C(\Z \cup \{ \pm \infty\}, B ) \to C(\Z \cup \{ \pm \infty\}, B ) / C_0(\Z, B) \cong B\oplus B$ can be concretely realised
by evaluating a function $f \in C(\Z \cup \{\pm \infty\}, B)$ at the endpoints $\pm \infty$. This induces a $\ast$-homomorphism
\[
 ({\rm ev}_L, {\rm ev}_R)\colon\ C(\Z \cup \{ \pm \infty\}, B ) \rtimes_{\wt{\alpha}} \Z \to (B \rtimes_\alpha \Z)^{\oplus 2}
\]
whose kernel is $C_0(\Z, B) \rtimes_\alpha \Z$.
Summarising our results, we obtain the following short exact sequence
\begin{equation} \label{eq:doublesided_SES}
 0 \to \calK \otimes B \to C(\Z \cup \{ \pm \infty\}, B ) \rtimes_{\wt{\alpha}} \Z \xrightarrow{({\rm ev}_L, {\rm ev}_R)} (B \rtimes_\alpha \Z)^{\oplus 2} \to 0.
\end{equation}

We remark that an analogous short exact sequence was constructed in~\cite[Section 4.2]{RichardSuzuki18} for the
case that $B = M_2(\C)$.

\subsection{The index formula}

To state our index formula, we now consider the setting where $B$ has a faithful norm-lower semicontinuous
trace $\tau_B\colon B \to \C$ such that $\tau(\one) = 1$ and $\tau_B\bigl( \alpha(b) \bigr) = \tau_B(b)$ for all $b \in B$.
If $B$ is a subalgebra of $M_n(\C)$ and we are in the setting of regular quantum walks, $\tau_B$ is the normalised
matrix trace.

Because $\tau_B$ is invariant under the automorphism
$\alpha \in \Aut(B)$, we can define the semifinite dual trace on the crossed product
\[
 \Tr_\tau\colon\ \Dom(\Tr_\tau) \subset B \rtimes_\alpha \Z \to \C, \qquad \Tr_\tau( S^n b) = \delta_{n,0} \tau_B(b).
\]

Let us recall the noncommutative calculus and winding number.
The algebra $B\rtimes_\alpha \Z$ possesses a derivation $\delta\colon \Dom(\delta) \to B\rtimes_\alpha \Z$, where
$\Dom(\delta)$ is dense, $S^n b \in \Dom(\delta)$ and $\delta( S^n b) = n S^n b$. It can be easily checked that
\[
 \delta(a_1 a_2) = \delta(a_1) a_2 + a_1 \delta(a_2), \qquad \delta(a_1^*) = - \delta(a_1)^*, \qquad
 \Tr_\tau( \delta( a_1 ) ) = 0
\]
for $a_1$, $a_2$ in the dense $\ast$-subalgebra $\calA \subset B\rtimes_\alpha \Z$ of elements $\sum_n S^n b_n$ such that the
function $n \mapsto \| b_n \|_B$ has superpolynomial decay in $\Z$.

\begin{Remark}
The noncommutative calculus has been highly successful in characterising the topological phase of systems in condensed
matter physics, which often admit a description via crossed product $C^*$-algebras, see~\cite{PSBbook} for example.
\end{Remark}

For invertible elements in a dense subspace of $B \rtimes_\alpha \Z$, we define the noncommutative winding number
\[
 \operatorname{Wind}\colon\ \Dom( \operatorname{Wind}) \subset GL( B\rtimes_\alpha \Z ) \to \C, \qquad
 \operatorname{Wind}( a) = \frac{1}{2{\rm i}\pi } \Tr_\tau\bigl( a^{-1} \delta(a) \bigr).
\]

If $B$ is a subalgebra of $M_n(\C)$ for some $n \in \mathbb{N}$, then
 $B \rtimes_\alpha \Z$ is isomorphic to a subalgebra of~$M_n( C(\T) )$ and $\operatorname{Wind}$ is the usual
 winding number of invertible matrix-valued functions of $\T$.

The noncommutative winding number of $B\rtimes_\alpha \Z$ is closely connected to the Toeplitz extension
\begin{gather} \label{eq:Toeplitz_ext}
 0 \to \calK \otimes B \to \calT_\alpha \to B \rtimes_\alpha \Z \to 0,
\end{gather}
where $\calT_\alpha$ is generated by $B$ and an isometry $\wh{S}$ such that
\[
 \wh{S} b = \alpha(b) \wh{S}, \qquad \wh{S}^* b = \alpha^{-1}(b) \wh{S}^*,
 \qquad \wh{S}^* \wh{S} = \one, \qquad \wh{S} \wh{S}^* = \one - p,
\]
with $p=p^*=p^2$ a projection. The surjection $\sigma\colon \calT_\alpha \to B \rtimes_\alpha \Z$ is uniquely
determined by $\sigma\big( \wh{S}b\big) = Sb$.

The short exact sequence in equation \eqref{eq:doublesided_SES} and the Toeplitz extension of equation \eqref{eq:Toeplitz_ext} give boundary maps in $K$-theory, $\partial_{L/R}$ and
$\partial_{{\rm Toep}}$, respectively. That is, there are homomorphisms
\[
 \partial_{L/R}\colon\ K_\ast(B\rtimes_\alpha \Z)^{\oplus 2} \to K_{\ast-1}( B), \qquad
 \partial_{{\rm Toep}}\colon\ K_\ast( B \rtimes_\alpha \Z) \to K_{\ast - 1}( B).
\]
The following result relates these two maps.

\begin{Lemma} \label{Lemma:double_limit_is_double_Toeplitz}
If $[(w_L,w_R)] \in K_\ast ( B\rtimes_\alpha \Z )^{\oplus 2}$, then
\[
 \partial_{L/R} \big[(w_L,w_R) \big] = \partial_{{\rm Toep}} [w_L] - \partial_{{\rm Toep}} [w_R],
\]
where $-[v]$ denotes the addition of $[v]^{-1}$.
\end{Lemma}
\begin{proof}
The fixed point $-\infty$ of the translation action of $\Z$ on $C(\Z \cup \{- \infty\}, B)$ gives a completely positive map
$B\rtimes_\alpha \Z \to C( \Z \cup \{- \infty\}, B) \rtimes_{\wt{\alpha}} \Z$ that is equivalent to the
completely positive map $Sb \mapsto \wh{S}b$ from the Toeplitz extension. That is, the Toeplitz extension is equivalent to the short exact sequence
\[
 0 \to \calK \otimes B \to C( \Z \cup \{-\infty\}, B) \rtimes_{\wt{\alpha}} \Z \xrightarrow{{\rm ev}_L} B \rtimes_\alpha \Z \to 0.
\]
The limit $+\infty$ reverses the orientation and so represents the inverse of the Toeplitz extension.
\end{proof}

Suppose that $u \in C( \Z \cup \{\pm \infty\}, B) \rtimes_{\wt{\alpha}} \Z$ is a chiral unitary on $\ell^2(\Z, B)$ with respect to some
$\gamma_0 \in C( \Z \cup \{ \pm \infty\}, B) \rtimes_{\wt{\alpha}} \Z$. If $u$ is of Fredholm type, then it determines
a class $\big[\calC(u) \big] \in KK(\C, B)$. Because $B$ possesses a unital trace $\tau_B$, we can also consider the numerical
in\-dex~$(\tau_B)_\ast \bigl( [\calC(u)] \bigr) \in \R$ (cf.~Section~\ref{subsec:induced_trace}). As the following result shows,
the noncommutative calculus on
$B \rtimes_\alpha \Z$ can be used to give a concrete formula for this index.

\begin{Theorem} \label{Theorem:doublesided_windingno}
Let $\bullet \in \{ L, R\}$ and $F \in C(\Z \cup \{\pm \infty\}, B) \rtimes_{\wt{\alpha}} \Z \subset \End_B\bigl(\ell^2(\Z, B)\bigr)$ be any element such that
$F_\bullet := {\rm ev}_\bullet(F) \in B\rtimes_\alpha \Z$ is unitary and is contained in $\Dom(\operatorname{Wind})$.
Then $F$ defines a~class $[F] \in KK(\C, B)$ and
\[
 ( \tau_B)_\ast \bigl( [F] \bigr) = \operatorname{Wind}(F_R) - \operatorname{Wind}(F_L)
 = \frac{1}{2{\rm i}\pi} ( \Tr_\tau( F_R^* \delta(F_R)) - \Tr_\tau( F_L^* \delta(F_L)) ) .
\]
\end{Theorem}
\begin{proof}
If $F$ is such that $F_L$ and $F_R$ are unitary, then recalling that $\mathbb{K}_B\bigl(\ell^2(\Z, B)\bigr) \cong \calK \otimes B$,
we have that $F \in C(\Z \cup \{\pm \infty\}, B) \rtimes_{\wt{\alpha}} \Z \subset \End_B\bigl(\ell^2(\Z, B)\bigr)$
is unitary modulo compact operators. Hence, the triple
\[
 \biggl( \C, \ell^2(\Z, B) \oplus \ell^2(\Z, B), \begin{pmatrix} 0 & F^* \\ F & 0 \end{pmatrix} \biggr)
\]
is an even Kasparov module and gives a class $[F] \in KK(\C, B) \cong K_0(B)$.
Because $F_L$ and $F_R$ are unitary elements in $B \rtimes_\alpha \Z$, we can consider the
class $[(F_L, F_R)] \in K_1( B \rtimes_\alpha \Z)^{\oplus 2}$.
Recalling the
index map in $K$-theory (see, for example, \cite[Section 8.3]{Blackadar}), $[F]$ represents
$\partial_{L/R} [( F_L, F_R)] \in K_0( B)\cong KK(\C, B)$
by construction.
Therefore, $(\tau_B)_\ast ([F]) = \bigl((\tau_B)_\ast \circ \partial_{L/R} \bigr) [( F_L, F_R)]$
and we consider the composition
\[
 K_1( B\rtimes_\alpha \Z)^{\oplus 2} \xrightarrow{ \partial_{L/R}} K_0(B) \xrightarrow{( \tau_B)_\ast} \C.
\]

Using Lemma \ref{Lemma:double_limit_is_double_Toeplitz}, we can write this map as
\[
 ( \tau_B)_\ast \bigl( \partial_{{\rm Toep}}( [F_L] ) \bigr) - ( \tau_B)_\ast \bigl(\partial_{{\rm Toep}}( [F_R] ) \bigr).
\]
Now, using~\cite[Proposition 3.3]{BKR1}, the boundary map $\partial_{{\rm Toep}}$ can be represented as
the Kasparov product with the unbounded Kasparov module $\big[\hat{X}\big] \in KK^1( B\rtimes_\alpha \Z, B)$ constructed in~\cite[Section~3]{BKR1}.
So we can equivalently consider the product
\begin{equation} \label{eq:boundary_to_Kasprod}
 K_1( B \rtimes_\alpha \Z) \times KK^1(B\rtimes_\alpha \Z, B ) \to K_0( B ) \xrightarrow{( \tau_B)_\ast} \C.
\end{equation}
If $F_L, F_R \in \Dom(\operatorname{Wind})$,
we can use the index formula~\cite[Theorem 6]{BSBMatrix}, which says that~\eqref{eq:boundary_to_Kasprod} applied
to a unitary $F_\bullet \in \Dom(\operatorname{Wind})$ is given by $-\operatorname{Wind}(F_\bullet)$.
The result now follows.
\end{proof}

\begin{Remarks}\quad
\begin{itemize}\itemsep=0pt
 \item[(1)] Given the setting of Theorem \ref{Theorem:doublesided_windingno}, if in addition the operator
 $F \in C(\Z \cup \{\pm \infty\}, B) \rtimes_{\wt{\alpha}} \Z$
has closed range, then $\Ker(F)$ and $\Ker(F^*)$ are finitely generated and projective modules over the unital $C^*$-algebra $B$.
In such a setting, the map $KK(\C, B) \xrightarrow{(\tau_B)_*} \C$ can be explicitly characterised and our index formula can be
written as
\[
 \tau_B \bigl( \Ker(F) \bigr) - \tau_B \bigl( \Ker(F^*) \bigr) = \operatorname{Wind}(F_R) - \operatorname{Wind}(F_L).
\]

 \item[(2)] There is a dense $\ast$-subalgebra $\calA \subset B \rtimes_\alpha \Z$ that is Fr\'{e}chet, stable under the
 holomorphic functional calculus and $GL(\calA) \subset \Dom(\operatorname{Wind})$. The noncommutative
 winding number is a homotopy invariant and gives a well-defined map $K_1( \calA) \to \C$.
 Because $K_1(\calA) \cong K_1( B \rtimes_\alpha \Z)$,
 we can define $\operatorname{Wind}(v)$ for any invertible $v \in B \rtimes_\alpha \Z$ via
 $\operatorname{Wind}(v) = \operatorname{Wind}(\wt{v})$, where $[v] = [\wt{v}] \in K_1(B \rtimes_\alpha \Z)$
 and $\wt{v} \in \calA$. Hence we can remove the assumption that
 $F_L, F_R \in \Dom(\operatorname{Wind})$ from Theorem \ref{Theorem:doublesided_windingno}, though the existence
 of elements $\wt{F}_L, \wt{F}_R \in \calA$ is non-constructive in general.
 \item[(3)] Given Lemma \ref{Lemma:double_limit_is_double_Toeplitz}, Theorem \ref{Theorem:doublesided_windingno}
 is not so surprising to those familiar with the Noether--Toeplitz index theorem. However, the use of double-sided limits
 is an interesting variation, something that also appears in the work of Matsuzawa~\cite{Matsuzawa}. We expect further
 variants and generalisations of such an index formula to hold that may also be of relevance for topological properties of chiral
 unitaries and split-step quantum walks. We leave this question to another place.
\end{itemize}
\end{Remarks}

Let us now examine the hypotheses of Theorem \ref{Theorem:doublesided_windingno} for the case of a
chiral unitary $u \in C( \Z \cup \{\pm \infty\}, B) \rtimes_{\wt{\alpha}} \Z$ with respect to
$\gamma_0 \in C( \Z \cup \{\pm \infty\}, B) \rtimes_{\wt{\alpha}} \Z$ and acting on $\ell^2(\Z, B)$.
If $\| \one - (u_L, u_R) \|_{B \rtimes_\alpha \Z} < 2$, then $u$ is of Fredholm type and
we obtain a Fredholm operator $\chi( \calC(u) )$ with $\chi$ a normalising function for
$\calC(u)$. Recalling that $({\rm i} + \calC(u))^{-1} = -\tfrac{{\rm i}}{2}( \one -u) \in C( \Z \cup \{\pm \infty\}, B) \rtimes_{\wt{\alpha}} \Z$,
we also have that $\chi (\calC(u) ) \in C( \Z \cup \{\pm \infty\}, B) \rtimes_{\wt{\alpha}} \Z$.

Because $\gamma_0 \in C( \Z \cup \{\pm \infty\}, B) \rtimes_{\wt{\alpha}} \Z$, so is
$p_0 = \tfrac{1}{2}(\one +\gamma_0)$ and $\one - p_0$. In particular the operator
$F = (\one - p_0) \chi( \calC(u) ) p_0 \in C( \Z \cup \{\pm \infty\}, B) \rtimes_{\wt{\alpha}} \Z$
is such that $F_L, F_R \in B\rtimes_\alpha \Z$ is unitary.
If $F_L$ and $F_R$ are contained in $\Dom(\operatorname{Wind})$, then we can say that
\[
 (\tau_B)_\ast \bigl( [\calC(u)] \bigr) = \operatorname{Wind}(F_R) - \operatorname{Wind}(F_L).
\]
The more difficult question is whether $F_L$ and $F_R$ are contained in $\Dom(\operatorname{Wind})$.
Certainly if $F$ is of the form
\begin{equation} \label{eq:smooth_elements}
 \sum_{n \in \Z} \wt{S}^n f_n, \qquad f_n \in C(\Z \cup \{\pm \infty\}, B),
\end{equation}
where $f_n$ is such that the function $n \mapsto \| f_n( \pm \infty) \|_B$ has superpolynomial decay in $n$,
then $F_L$ and $F_R$ will be in $\Dom(\operatorname{Wind})$.
However, a concrete expression for $F$ is challenging in general.

Let us consider the case that $\one - u \in C_0(\Z, B) \rtimes_{\wt{\alpha}} \Z$, i.e.,
$u$ is in the minimal unitisation $\mathbb{K}_B\bigl( \ell^2( \Z, B) \bigr)^\sim$.
Then $\calC(u)$ has compact resolvent we can use the normalising function
$\chi(x) = x\bigl(1+x^2\bigr)^{-1/2}$. Recalling \eqref{eq:C(u)_bdd_transfrom} on p.~\pageref{eq:C(u)_bdd_transfrom},
\begin{align*}
 F = (\one - p_0) F_{\calC(u)} p_0 = \frac{{\rm i}}{2} (\one - p_0) (\one + u)V p_0,
\end{align*}
where $V$ is the extension of the operator on $(u-\one)\ell^2(\Z, B)$,
$V(u-\one)e = |u-\one|e$ for $e \in \ell^2(\Z, B)$. Supposing that $u$ and
$\gamma_0$ are of the form described in~\eqref{eq:smooth_elements}, then
the only obstruction is the operator~$V$. Writing
\[V = |\one -u|(\one -u)^{-1} = (2-u-u^*)^{1/2}(\one -u)^{-1},\] we have that $V \in C( \Z \cup \{\pm \infty\}, B) \rtimes_{\wt{\alpha}} \Z$,
but a condition on $u$ that implies sufficient regularity of~$V_L$ and~$V_R$ with respect
to $\delta$ and $\Tr_\tau$ is difficult to state
in general. We leave a more comprehensive analysis of this question to future work.

\subsection*{Acknowledgements}

The author thanks C.~Cedzich, S.~Richard and Y.~Tanaka for helpful discussions.
We also thank the anonymous referees for their numerous suggestions that have helped improve the manuscript.
This work is supported by a JSPS Grant-in-Aid for Early-Career Scientists (No.~19K14548).
An earlier version of this paper was completed while the author was affiliated to the WPI-AIMR, Tohoku University.

\pdfbookmark[1]{References}{ref}
\LastPageEnding

\end{document}